\documentclass[12pt]{amsart} 
\usepackage{verbatim, latexsym, amssymb, amsmath,color}
\usepackage{enumitem}
\usepackage{epsfig}

\makeatletter
\renewcommand{\tocsection}[3]{%
  \indentlabel{\@ifnotempty{#2}{\bfseries\ignorespaces#1 #2\quad}}\bfseries#3}
\renewcommand{\tocsubsection}[3]{%
  \indentlabel{\@ifnotempty{#2}{\ignorespaces#1 #2\quad}}#3}

\newcommand\@dotsep{4.5}
\def\@tocline#1#2#3#4#5#6#7{\relax
  \ifnum #1>\c@tocdepth 
  \else
    \par \addpenalty\@secpenalty\addvspace{#2}%
    \begingroup \hyphenpenalty\@M
    \@ifempty{#4}{%
      \@tempdima\csname r@tocindent\number#1\endcsname\relax
    }{%
      \@tempdima#4\relax
    }%
    \parindent\z@ \leftskip#3\relax \advance\leftskip\@tempdima\relax
    \rightskip\@pnumwidth plus1em \parfillskip-\@pnumwidth
    #5\leavevmode\hskip-\@tempdima{#6}\nobreak
    \leaders\hbox{$\m@th\mkern \@dotsep mu\hbox{.}\mkern \@dotsep mu$}\hfill
    \nobreak
    \hbox to\@pnumwidth{\@tocpagenum{\ifnum#1=1\bfseries\fi#7}}\par
    \nobreak
    \endgroup
  \fi}
\AtBeginDocument{%
\expandafter\renewcommand\csname r@tocindent0\endcsname{0pt}
}
\def\l@subsection{\@tocline{2}{0pt}{2.5pc}{5pc}{}}
\makeatother

\usepackage{geometry}
\usepackage{hyperref}

\DeclareMathOperator{\II}{II}

\DeclareMathOperator{\Rm}{Rm}
\DeclareMathOperator{\Ric}{Ric}

\DeclareMathOperator{\Hess}{Hess}

\DeclareMathOperator{\diam}{diam}

\DeclareMathOperator{\Vol}{Vol}

\DeclareMathOperator{\dist}{dist}
\DeclareMathOperator{\injrad}{injrad}
\DeclareMathOperator{\rad}{rad}

\DeclareMathOperator{\minvol}{MinVol}

\DeclareMathOperator{\essminvol}{ess-MinVol}

\DeclareMathOperator{\e}{\textbf{e}}

\DeclareMathOperator{\Aut}{Aut}

\newtheorem{theo}{Theorem}[]
\newtheorem{prop}[theo]{Proposition}
\newtheorem{prop/def}[theo]{Proposition/Definition}
\newtheorem{theo/def}[theo]{Theorem/Definition}
\newtheorem{lemme}[theo]{Lemma}
\newtheorem{definition}{Definition}[section]

\begin{document}
\title[]
{Essential minimal volume of Einstein 4-manifolds}

\author{Antoine Song}
\address{California Institute of Technology\\ 177 Linde Hall, \#1200 E. California Blvd., Pasadena, CA 91125}
\email{aysong@caltech.edu}

\maketitle

\begin{abstract} 
The minimal volume of a  closed manifold $M$ is the infimum of the volume of $(M,g)$ over all metrics $g$ with sectional curvature between $-1$ and $1$. 
We introduce a variant called the essential minimal volume, $\essminvol(M)$, which is the limit, as $\delta>0$ goes to $0$, of the infimum of the volume of the $\delta$-thick part of $(M,g)$ over all metrics $g$ with sectional curvature between $-1$ and $1$. We show that, for some universal constant $C>0$, any closed Einstein 4-manifold $M$ with Euler characteristic $\e(M)$ satisfies
$$C^{-1}\e(M) \leq \essminvol(M) \leq C\e(M).$$
As a corollary, these inequalities are true for the essential minimal volume of  closed complex surfaces of nonnegative Kodaira dimension.
We conjecture that those linear bounds  in fact  hold for the minimal volume.



\end{abstract}


\vspace{1em}

\section*{Introduction}

Perhaps the simplest way of measuring the complexity of a closed $n$-manifold $M$, from a Riemannian point of view, is to consider its \emph{minimal volume}. This invariant, considered by Gromov in \cite{Gromov82}, is defined as
$$\minvol(M):=\inf\{\Vol(M,g); \quad \text{$g$ satisfies }|\sec_g|\leq 1\}.$$
Quite paradoxically, the value of minimal volume is only known for a handful of manifolds, in spite of its simple definition. 
It can be explicitly computed for surfaces \cite{Gromov82}, and also for negatively curved locally symmetric spaces by the beautiful work of Besson-Courtois-Gallot \cite{BCG95,BCG96}.
In general, it seems more fruitful to aim at coarse bounds for the minimal volume of natural families of manifolds. For closed 3-dimensional manifolds, good coarse bounds are already known (this is essentially contained in \cite{Souto01}), so we focus on $4$-dimensional manifolds. To simplify our task, we introduce a more tractable variant of the minimal volume. If $\delta>0$ and $g$ is a metric on $M$, let $M^{(g)}_{>\delta}$ be the subset of $M$ where the injectivity radius of $g$ is larger than $\delta$. We define the \emph{essential minimal volume} of $M$ as
$$\essminvol(M):= \lim_{\delta\to 0} \quad \inf\{\Vol(M^{(g)}_{>\delta},g); \quad  \text{$g$ satisfies }|\sec_g|\leq 1\}.$$


\vspace{0.5em}

\begin{theo}\label{main2}
There is a constant $C>0$ such that the following holds. Let $M$ be any closed $4$-manifold admitting an Einstein metric, or any closed complex surface with nonnegative Kodaira dimension. Then
$$C^{-1}\e(M)\leq \essminvol(M)\leq C \e(M),$$
where $\e(M)$ is the Euler characteristic of $M$.
\end{theo}

Theorem \ref{main2} will be shown in Theorem \ref{mainineq} and Theorem \ref{blip}. 
A key feature of Theorem \ref{main2} is that no assumption is made on the Einstein metric.
The main proof ingredients are the collapsing theory of Cheeger, Fukaya and Gromov  \cite{CheegerGromov86} \cite{CheegerGromov90} \cite{Fukaya87} \cite{Fukaya88} \cite{Fukaya89} \cite{CFG92},  the codimension $4$ theorem  of Cheeger-Naber \cite{CheegerNaber15}, a thick-thin decomposition for Einstein 4-manifolds and a covering argument (Section \ref{section4}).

Previously, in a series of papers \cite{LeBrun96,LeBrun99,LeBrun01}, LeBrun used gauge theoretic methods to compute certain integral minimal volumes for complex surfaces involving the scalar and Ricci curvatures, and also studied some minimal volumes defined with the sectional curvature. Theorem \ref{main2} provides new information for complex surfaces, related to sectional curvature. As for Einstein $4$-manifolds, Anderson \cite{Anderson90}, Bando-Kasue-Nakajima \cite{BKN89}, Anderson-Cheeger \cite{AndersonCheeger91}, and more recently Cheeger-Naber \cite{CheegerNaber15} showed various finiteness results under non-collapsing and diameter bound assumptions. 
In some sense, Theorem \ref{main2} can be viewed as a statement quantifying those earlier results.

For a family of closed manifolds for which the essential minimal volume is uniformly bounded, the number of smooth structures is bounded modulo regions carrying a special topological structure called $F$-structure. This property follows from
compactness results for bounded sectional curvature metrics and collapsing theory.
Hence, Theorem \ref{main2} is relevant to a question of Kotschick \cite{Kotschick98} (see also Ishida-LeBrun \cite{IshidaLeBrun02}) which asks whether or not a given closed manifold only supports finitely many smooth structures with an Einstein metric.

We conjecture that the main theorem is true for the minimal volume (the nontrivial inequality being the upper bound). Theorem \ref{main2} is also related to a potential new topological obstruction for $4$-dimensional Einstein metrics.
For more detailed comments and other questions,  see Subsections \ref{crac} and \ref{croc}.

\vspace{2em}

\subsection*{Acknowledgements} 
I am grateful to John Lott for many interesting discussions. I would also like to thank Song Sun, Aaron Naber, Xiaochun Rong, Ruobing Zhang, Luca Di Cerbo, Ben Lowe, Guofang Wei for helpful conversations, and Claude LeBrun, Zolt\'{a}n Szab\'{o} for comments. 
\newline

This research was conducted during the period the author served as a Clay Research Fellow.
A.S. was partially supported by NSF grant DMS-2104254.

\vspace{3em}

\section{Definitions and first properties}


\subsection{Thick/thin decomposition}\label{subsection:tt}
Let $M$ be a connected (smooth) closed $n$-dimensional manifold
In a manifold with bounded sectional curvature, neighborhoods of points with a lower bound on the injectivity radius (denoted by $\injrad$) have well controlled geometry. More precisely, the work of Cheeger \cite{Cheeger70}, Gromov \cite{Gromov81}, Peters \cite{Peters87}, Greene and Wu \cite{GreeneWu88} imply the following. Fix $\alpha\in(0,1)$. If $(M_i,g_i)$ is a sequence with $|\sec_{g_i}|\leq 1$ and $p_i\in M_i$, if moreover $\injrad(p_i)\geq 1$ then $(M_i,g_i,p_i)$ subsequentially converges in the pointed Gromov-Hausdorff topology to a complete pointed Riemannian manifold $(M_\infty, g_\infty,p_\infty)$, where $g_\infty$ is a $C^{1,\alpha}$-metric.

On the other hand, in a manifold with bounded sectional curvature, neighborhoods of points with small injectivity radius $\injrad$ are described by the collapsing theory developed by Cheeger-Gromov \cite{CheegerGromov86,CheegerGromov90}, Fukaya \cite{Fukaya87,Fukaya88,Fukaya89}, Cheeger-Fukaya-Gromov \cite{CFG92}. Given $(M,g)$ with $|\sec_g|\leq 1$, there is a dimensional constant $\epsilon_n$ such that a neighborhood of the $\epsilon_n$-thin part
$$M^{(g)}_{\leq \epsilon_n} := \{q\in M;\quad \injrad_g(q) \leq \epsilon_n\}$$
carries an $F$-structure of positive rank and also an $N$-structure consistent with the $F$-structure. The complement of the $\epsilon_n$-thin part is called the $\epsilon_n$-thick part. More details on $F$-structures and $N$-structures are given below.

\vspace{1em}

\subsection{F-structures and N-structures}\label{subsection:fn}


If $\mathfrak{g}$ is a sheaf of Lie group on a manifold $N$ with associated sheaf of Lie algebras $\underline{\mathfrak{g}}$, an action $h$ of $\mathfrak{g}$ on $N$ is by definition a homomorphism from $\underline{\mathfrak{g}}$ to the sheaf of smooth vector fields. Recall that a subset $S$ of $N$ is said to be saturated (or invariant) for the action $h$ if the following holds: for any $x\in S$ and any open set $U\subset N$ containing $x$, for any curve $c$ that starts at a point of $S$ and that is everywhere tangent to the image $h(\mathcal{X})$ of some section $\mathcal{X}\in \mathfrak{g}(U)$, $c$ remains included inside $S$. The unique minimal saturated subset containing a point $x$ is called the orbit of $x$ and will be denoted by $\mathcal{O}^\mathfrak{g}_x$.

\textbf{$F$-structures.} 
An $F$-structure is a generalization of a torus action \cite{CheegerGromov86}. An $F$-structure on a manifold $N$ is determined by the following data. We start with a sheaf of tori, $\mathfrak{t}$, and the associated sheaf of Lie algebras $\underline{\mathfrak{t}}$, together with an action $h$ of $\mathfrak{t}$ on $N$. We then suppose that for each point $x\in N$, there is a saturated open neighborhood $V(x)$ and a finite normal covering $\pi:\tilde{V}(x) \to V(x)$ such that 
\begin{itemize}
\item the action $h$ of $\mathfrak{t}$ induces an action $\pi^*(h)$ of the pullback sheaf $\pi^*(\mathfrak{t})$, which is the infinitesimal generator of an effective action of the torus $\pi^*(\mathfrak{t})(\tilde{V}(x))$ on $\tilde{V}(x)$,
\item for all open set $W\subset \tilde{V}(x)$ such that $W \cap \pi^{-1}(x) \neq \varnothing$, the structure homomorphism $\pi^*(\mathfrak{t})(\tilde{V}(x)) \to \pi^*(\mathfrak{t})(W)$ is an isomorphism,
\item the neighborhood $V(x)$ and covering $\tilde{V}(x)$ can be chosen independent of $x$, for $x\in \mathcal{O}^\mathfrak{t}_x \subset N$.
\end{itemize}
From the above definition of $F$-structures, we see that the orbits are given by the projected images in $V(x)$ of the orbits of the local torus action in $\tilde{V}(x)$, which are quotients of tori embedded in $V(x) \subset N$. An $F$-structure has positive rank if all its orbits have positive dimensions. It is polarized if in each $V(x)$ as above, the dimension of the orbits is constant.

There is an equivalent and more practical way of defining $F$-structures by using the notion of atlas. An atlas for an $F$-structure is given by a locally finite cover $\{V_{\alpha}\}$ by open sets, and for each $\alpha$ there is a finite normal covering  $\pi_\alpha : \tilde{V}_{\alpha} \to V_{\alpha}$ with covering group $\Gamma$, together with an effective action of $\Gamma \ltimes_\rho T^{k_\alpha}$ on $\tilde{V}_\alpha$, for some representation $\rho : \Gamma \to \Aut(T^{k_\alpha})$. The local torus actions are required to satisfy the following compatibility conditions: if $U_\alpha \cap U_\beta \neq \varnothing$, then $\pi_\alpha^{-1}(V_\alpha \cap V_\beta)$ and $\pi_\beta^{-1}(V_\alpha\cap V_\beta)$ have a common covering space on which the torus actions lift to actions of $T^{k_\alpha}$ and $T^{k_\beta}$, so that $T^{k_\alpha}$ is a subgroup of $T^{k_\beta}$ or vice versa.

If a Riemannian $n$-manifold $(M,g)$ has sectional curvature between $-1$ and $1$, then a neighborhood of the $\epsilon_n$-thin part is a union of orbits of an $F$-structure of positive rank. Conversely, a manifold endowed with an $F$-structure of positive rank has a sequence of smooth metrics with bounded sectional curvature and injectivity radius going to $0$. These results are proved in \cite{CheegerGromov86,CheegerGromov90}. Paternain and Petean \cite{PaternainPetean03} later showed that when there is an $F$-structure (non-trivial, but not necessarily of positive rank), one can collapse the manifold with a uniform lower bound on the sectional curvature.

When $M$ carries an $F$-structure of positive rank which is polarized, one can collapse the volume of $M$ while keeping the curvature bounded, but the converse is probably not true (see \cite[Remark 4.1]{CheegerGromov86}). In dimensions at most $4$ however, volume collapsing with bounded curvature does imply the existence of a polarized $F$-structure; in dimension $4$ this is due to Rong \cite{Rong93}. In higher dimensions, Cheeger-Rong \cite{CheegerRong96} showed that for a closed manifold $(M,g)$ such that the diameter is bounded by $D$ and the sectional curvature is bounded by $1$, then it carries a polarized $F$-structure if the volume is smaller than a constant depending on $D$ and the dimension.

\textbf{$N$-structures.}
An $F$-structure as above describes collapsed manifolds on the scale of the injectivity radius. More generally, the thin part of a manifold with bounded curvature has an $N$-structure \cite{CFG92}, which generalizes the notion of $F$-structures, and which captures all the collapsed directions. $N$-structures on a manifold $N$ are determined by data similar to that defining $F$-structures. 
We start with a sheaf $\mathfrak{n}$ of simply connected nilpotent Lie group on $N$, with an action $h$. We suppose that for each $x\in N$, there is an open neighborhood $U(x)$ saturated for $h$ and a normal covering $\pi:\tilde{U}(p) \to U(p)$ 
such that
\begin{itemize}
\item the action $h$ of $\mathfrak{n}$ induces an action $\pi^*(h)$ of the pullback sheaf $\pi^*(\mathfrak{n})$, which is the infinitesimal generator of an action of the nilpotent group $\pi^*(\mathfrak{t})(\tilde{U}(p))$ on $\tilde{U}(p)$ whose kernel $K_p$ is discrete,
\item for all open set $W\subset \tilde{U}(p)$ such that $W \cap \pi^{-1}(x) \neq \varnothing$, the structure homomorphism $\mathfrak{n}(\tilde{U}(p)) \to \mathfrak{n}(W)$ is an isomorphism,
\item the neighborhood $U(p)$ and covering $\tilde{U}(p)$ can be chosen independent of $p$, for $p\in \mathcal{O}^\mathfrak{n}_x \subset N$.
\end{itemize}
Cheeger, Fukaya and Gromov \cite{CFG92} showed that given $\epsilon>0$ there exist $\rho>0$ and $k>0$ so that if a complete Riemannian $n$-manifold $(M,g)$ has sectional curvature between $-1$ and $1$, then any open set $Y'$ inside the $\varepsilon(n,\epsilon)$-thin part of $(M,g)$ is contained in an open set $Y$ carrying an $N$-structure $\mathfrak{n}$ with the following special properties. There is a metric $g_\epsilon$ on $M$ with 
\begin{itemize}
\item $e^{-\epsilon} g < g_\epsilon < e^{\epsilon} g$,
\item $|\nabla^g - \nabla^{g_\epsilon}|<\epsilon$,
\item $|(\nabla^{g_\epsilon})^i \Rm_{g_\epsilon} |< C(n,i,\epsilon)$.
\end{itemize}
Moreover, for any point $p$ in $Y$, there is an open neighborhood $V(p)$ and a normal covering $\pi:\tilde{V}(p)\to V(p)$ with covering group $\Lambda$ such that $\tilde{U}(p)\subset \tilde{V}(p)$ (see notations used in definition of $N$-structures above) and
\begin{enumerate}
\item there is a Lie group $H$ generated by $\Lambda$ and its identity component $N_0 := \pi^*(\mathfrak{t})(\tilde{U}(p))/ K_p $ (see notations used in definition of $N$-structures above), such that the natural action of $H$ is isometric with respect to the metric induced by $g_\epsilon$ (the metric $g_{\epsilon}$ is called ``invariant''),
\item $V(p)$ contains the geodesic ball $B_{g_\epsilon}(p, \rho)$,
\item the $g_\epsilon$-injectivity radius on $\tilde{V}(p)$ is larger than $\rho$,
\item $\sharp(\Lambda /\Lambda \cap N_0) \leq  k$,
\item the orbits of the  $N$-structure are compact and have $g_\epsilon$-diameter less than $\epsilon$.
\end{enumerate}
Moreover, any orbit $\mathcal{O}$ of $\mathfrak{n}$ is finitely covered by a nilmanifold of the form $\Gamma \setminus N'$ where $\Gamma$ is a cocompact discrete subgroup (i.e. a lattice) in a simply connected nilpotent Lie group $N'$. Note that if $C(N')$ is the center of $N'$, then $\Gamma\cap C(N')$ is a lattice inside $C(N')$ and $\Gamma\setminus C(N')$ is a torus which projects to a torus quotient embedded inside $\mathcal{O}$. In fact, as noted in \cite[Remarque 1.9]{CFG92}, by considering the ``center'' of the $N$-structure $\mathfrak{n}$ associated to a collapsed region $Y$, we obtain an $F$-structure which is necessarily of positive rank, and which will be called the canonical $F$-structure $\mathfrak{t}$ coming from $\mathfrak{n}$. In this case, each $N$-orbit is then a disjoint union of $F$-orbits. \cite{CaiRong09,Eriksson-Bique18} are also useful references.


\subsection{Definition of the essential minimal volume}
$M$ is a closed manifold. For $\delta>0$ and for a metric $g$ on $M$, set
$$M^{(g)}_{>\delta} := \{x\in M; \quad \injrad_g(x) > \delta\},$$
$$\minvol_{>\delta}(M) := \inf\{\Vol(M^{(g)}_{>\delta},g);\quad |\sec_g|\leq 1\}.$$
Note that the quantity $\minvol_{>\delta}(M)$ is non-increasing with respect to $\delta$ and is bounded, so it converges as $\delta\to 0$.

\begin{definition}\label{essminvol}
The essential minimal volume of $M$ is 
$$\essminvol(M):= \lim_{\delta\to 0} \minvol_{>\delta}(M).$$
\end{definition}

 The name $\essminvol$ partly comes from the notion of $\epsilon$-essential volume defined in \cite{BGS85} in the context of negatively curved manifolds. 
 In some sense, $\essminvol(M)$ measures the smallest possible size of the thick part of a metric on $M$ with bounded sectional curvature. 
Since for any $\delta>0$, a neighborhood of the $\delta$-thick part of a manifold with bounded sectional curvature admits a smooth triangulation with a number of vertices bounded linearly by its volume and for which the degree of each vertex is uniformly bounded (all the constants depend on $\delta$), any upper bound on $\essminvol(M)$ in fact yields a  bound on a controlled triangulation of $M$ outside of regions covered by $F$-structures. Similarly, by compactness results for metrics with bounded sectional curvature \cite{Cheeger70} \cite{Gromov81} \cite{Peters87},\cite{GreeneWu88}, a bound on $\essminvol$ implies a bound on the number of smooth manifolds, up to regions saturated by $F$-structures.

\subsection{First properties}
\label{croc}

From the definitions, 
$$\essminvol(M)\leq \minvol(M)$$
and it can be checked that they are equal for closed oriented $2$-manifolds (and $3$-manifolds with more work).
Nevertheless, the difference between $\minvol$ and $\essminvol$ can be arbitrarily large in higher dimensions, as can be shown using \cite[Example 1.9]{CheegerGromov86}.

Not many lower bounds for $\minvol$ and $\essminvol$ are known. Thanks to \cite{Gromov82}, both are lower bounded by the ``simplicial volume'' of $M$, up to a dimensional constant. There is a simple bound for $\essminvol(M)$ in terms of the Euler characteristic $\e(M)$ of $M$ (see \cite{Gromov82} for bounds on the minimal volume using characteristic numbers):
\begin{lemme} \label{euler}
There is a dimensional constant $C_n$ such that
$$\e(M)\leq C_n \essminvol(M).$$
\end{lemme}

\begin{proof}
This is essentially a corollary of \cite{CheegerGromov91}. Consider some $\epsilon_n>0$. By definition of the essential minimal volume, there is a metric $g$ on $M$ with $|\sec_g|\leq 1$, so that the $\epsilon_n$-thick part $M^{(g)}_{\leq \epsilon_n}$ of $(M,g)$ has volume at most $C'_n \essminvol(M)$ for some dimensional constant $C'_n>0$.
By Lemma \ref{chopping}, which will be proved in the next Subsection, we can modify $g$ to get a new metric $g_1$ with $|\sec_{g_1}|\leq 1$ so that for some dimensional constant $C''_n>0$ and some neighborhood $U$ of $M^{(g)}_{\leq \epsilon_n}$ with smooth boundary,
$\Vol(U,g_1)\leq C''_n \essminvol(M)$, $g_1$ is a product metric on a uniform tubular neighborhood the boundary $\partial U$ with is $g_1$-totally geodesic, and $\injrad_{g_1} \leq C''_n\epsilon_n$ on $M\setminus U$.
Fix $\epsilon_n$ small enough so that \cite{CheegerGromov90} applies to the $C''_n\epsilon_n$-thin part of $(M,g_1)$.
Then $M\setminus U$ is saturated by an $F$-structure of positive rank so it has vanishing Euler characteristic by \cite[Lemma 1.5]{CheegerGromov86}.
We conclude by applying the Chern-Gauss-Bonnet formula for compact manifolds with boundary to $(U,g_1)$.

\end{proof}

We have the following  estimate for the essential minimal volume of connected sums:
\begin{lemme} \label{connected sum}
Let $M_1,M_2$ be two closed connected manifolds of same dimension $n$, and let $k,l$ be positive integers. Then the connected sum $kM_1\sharp l M_2$ satisfies
$$\essminvol(k M_1\sharp l M_2) \leq  C_n\big(k\essminvol(M_1) + l\essminvol(M_2)+k+l \big) $$
where $C_n>1$ is a dimensional constant.
\end{lemme}
\begin{proof}
It suffices to show the following observation: for any closed manifold $M$ and any $\delta>0$, there is a metric $g$ with $|\sec_g|\leq 1$, such that
$$\Vol(M_{>\delta}^{g},g) \leq C_n(\essminvol(M) +1)$$
for some large enough dimensional constant $C_n$, and such that there is a point $x\in M$ where the injectivity radius of $g$ is larger than $1$. Indeed, if this observation is true then the lemma follows by induction and by joining two manifolds $M_1$, $M_2$ using  a bounded geometry neck attached to the bounded geometry regions. 

The observation is only nontrivial when for any $\delta>0$, there is a metric $g$ such that the $\delta$-thick part $M_{>\delta}^{g}$ is empty. By \cite{CheegerGromov86} (see Subsections \ref{subsection:tt} and \ref{subsection:fn}), such a scenario necessarily implies  that $M$ admits an $F$-structure of positive rank, 
so there is a fiber of the $F$-structure, i.e. a $k$-dimensional torus quotient $K$ embedded in $M$, and a tubular neighborhood $\mathcal{N}$ of $K$ diffeomorphic to $K\times D^{n-k}$ where $D^{n-k}$ is a Euclidean unit ball in $\mathbb{R}^{n-k}$, 
such that both $\mathcal{N}$ and $M\backslash \mathcal{N}$ are saturated by the $F$-structure in a natural way. 

Recall that in each dimension, the number of diffeomorphism types of torus quotients, such as $K$, is finite. 
We start with a flat metric $g'=h_K\oplus h'$ on the closure of $\mathcal{N} \approx K\times D^{n-k}$ such that $|\sec_{g'}|\leq 1/2$ and $g'$ has injectivity radius at least $1$. Here, 
 $h_K$ is a flat metric on $K$ with volume bounded by a dimensional constant, and $h'$ is a metric on $\mathbb{D}^{n-k}$ which is the product of a unit round metric $g_{S^{n-k}} $ and the standard metric on $(0,1)$ near the boundary $\partial D^{n-k}$.
 Let $\delta>0$. 
By applying the metric deformation of \cite[Section 4]{CheegerGromov86} to $M\setminus \mathcal{N}$, we can find a metric $g_1$ on $M\setminus \mathcal{N}$ such that  $|\sec_{g_1}|\leq 1$, $\injrad_{g_1}< \delta$, and such that we have the product metric structure $g_2 =\lambda h_K\oplus h''$ near $\partial (M\setminus \mathcal{N})$, where $\lambda>0$ is some tiny constant and $h''$ is a product metric equal to $h'$ on the boundary.   
On the other hand, applying an elementary version of the metric deformation in \cite[Section 3]{CheegerGromov86} to contract the $K$-factors near the boundary $\partial \mathcal{N}$, we can construct a metric $g_2$ on $\mathcal{N}$ with $|\sec_{g_2}|\leq1$, such that there is a point in $\mathcal{N}$ with injectivity radius $1$, but $\injrad_{g_2}< \delta$ around the boundary $\partial \mathcal{N}$, and $g_2 = \lambda h_K\oplus h'$ near $\partial \mathcal{N}$. Moreover, we can ensure that
$$\Vol(\mathcal{N}^{g_2}_{>\delta}) \leq C_n$$
for some dimensional constant $C_n$. Those properties are verified if for instance the size of the $K$-factors are contracted exponentially with respect to the $g_2$-distance to the central $K$-factor $K\times \{0\}$ in $\mathcal{N}$. 
We conclude the proof of the observation by gluing $g_1$ and $g_2$ along $\partial \mathcal{N}$.
\end{proof}

Let $\mathcal{M}_{|\sec|\leq 1}(M)$ be the set of metrics on $M$ with sectional curvature between $-1$ and $1$.
As a remark, based on the convergence theory for bounded sectional curvature metrics \cite{Cheeger70, Gromov81, Peters87, GreeneWu88}, it can be shown that $\essminvol(M)$ is the infimum of the volume of $(Y,h)$, where $(Y,h)$ belongs to a natural closure $\overline{\mathcal{M}_{|\sec|\leq 1}}(M)$ of $\mathcal{M}_{|\sec|\leq 1}(M)$. Moreover, there is a minimizing metric, namely an $n$-manifold with a complete finite volume $C^{1,\alpha}$-metric $(M_\infty,g_\infty)$ belonging to $\overline{\mathcal{M}_{|\sec|\leq 1}}(M)$, such that
$$\essminvol(M) = \Vol(M_\infty,g_\infty).$$
Curiously, the analogous  property for the standard minimal volume is probably false, see the example of Januszkiewicz \cite[Example 1.9]{CheegerGromov86}.
\vspace{1em}

We end this subsections with some questions:
\begin{enumerate}
\item What are $\minvol$ and $\essminvol$ for the standard $4$-sphere? What do the minimizing metrics  look like? 
Note that according to \cite{Ivansic12} 
there are embedded $2$-tori inside $S^4$ such that their complement admits a hyperbolic metric with volume equal to that of the round unit $S^4$, so if the unit round metric on $S^4$ is a minimizer for the essential minimal volume of $S^4$, then it is non-unique. 
\item Some topological 4-manifolds admit an infinite number of different smooth structures.
 Is there a bound on $\minvol$ and $\essminvol$ depending only on the underlying topological structure of a 4-manifold?
 Generalizing to noncompact manifolds, for an exotic $\mathbb{R}^4$, can $\minvol$ and $\essminvol$ be zero or finite?
\end{enumerate}


\subsection{Uniform local collapsing in dimension 4}

The main result of this section is the technical Proposition \ref{keyprop} , which enables to locally collapse 4-manifolds with a uniform bound on the curvature.
In what follows, if $S$ is a subset of a manifold $(V,g)$ and $r>0$, set
$$T_{g,r}(S):=\{x\in M; \quad \dist_g(x,S)< r\}.$$
Recall that if $\mathbf{A}:=\{A_l\}_{l=1}^\infty$ is a sequence of positive numbers, a metric $g$ is said to be $\mathbf{A}$-regular if 
$$\forall l\geq 0,\quad |\nabla^{l}\Rm_g|_g \leq A_l.$$

We start with a corollary of the chopping techniques in  \cite{CheegerGromov91}:
\begin{lemme} \label{chopping}
Let $(M,g)$ be an $\mathbf{A}$-regular complete Riemannian $n$-manifold and let $S$ be a compact subset of $M$. 
Then there exist a constant $\kappa=\kappa(n,\mathbf{A})>0$, a sequence of positive numbers $\mathbf{B} = \mathbf{B}(n,\mathbf{A})$, and a new complete metric $\tilde{g}$ on $M$ such that 
\begin{itemize}
\item $\tilde{g}=g$ on $S$ and $\kappa g < \tilde{g} < \kappa^{-1} g$ on $M$,
\item $|\sec_{\tilde{g}}|\leq \kappa^{-1} |\sec_g|$ everywhere and $\tilde{g}$ is $\mathbf{B}$-regular,
\item there is a compact subset $U$ such that $S\subset U \subset T_{g,1}(S)$ and if $\partial U\neq \varnothing$, then with respect to $\tilde{g}$, the $\kappa$-tubular neighborhood of $\partial U$ is isometric to the product metric 
$$(\partial U \times (-\kappa, \kappa), \tilde{g}\big|_{\partial U} \oplus dt^2)$$ where $dt^2$ is the standard metric on $(-\kappa, \kappa)$. 
\end{itemize}

\end{lemme}
\begin{proof}
By the proof of the main result in \cite{CheegerGromov91}, there is a neighborhood $U$ of $S$ such that $S\subset U \subset T_{g,1}(S)$ and if $\partial U\neq \varnothing$, it satisfies
$$|\II_{\partial U}|_g < C \quad\text{and}\quad \underline{\rho}_g(\partial U) > C^{-1}$$
for some $C=C(n)$. Here $\II_{\partial U}$ denotes the second fundamental form of $\partial U$ and $\underline{\rho}_g(\partial U)$ is the normal injectivity radius of $\partial U$, all with respect to $g$. (See \cite[(1.11)]{CheegerGromov91} where $f_X$ should read $\nabla f_X$). Hence in a normal exponential coordinate chart of a neighborhood $W$ of $\partial U$, the metric $g$ takes the form 
$$\big(\partial U \times (-2\kappa_1,2\kappa_1), h(t) \oplus dt^2\big)$$ 
where $h(t)$ is a metric on $\partial U \times \{t\}$ and $\kappa_1\in (0,\frac{1}{2})$ only depends on $n$ and $\mathbf{A}$.
Consider a cut-off function 
$$\varphi : (-2\kappa_1,2\kappa_1) \to [0,1]$$ such that $\varphi(t) = 0$ near $t=-2\kappa_1$ and $t=2\kappa_1$, $\varphi(t)=1$ if $t\in (-\kappa_1,\kappa_1)$, and the $i$-th derivative of $\varphi$ is bounded in terms of $i$ and $\kappa_1$. Since $g$ is $\mathbf{A}$-regular, the $i$-th covariant derivatives of $h(t)$ on $\partial U \times \{t\}$ and the $i$-th $t$-derivatives of $h(t)$ are bounded in terms of $i$, $n$ and $\mathbf{A}$. Consider the new metric $\tilde{g}$ equal to $g$ outside the neighborhood $W$ of $\partial U$, and given by the following expression in the chart $\partial U \times (-2\kappa_1,2\kappa_1)$:
$$\tilde{g} = \big((1-\varphi(t))h(t) + \varphi(t) h(0) \big)\oplus dt^2.$$
It follows that 
 $$\kappa_2 g < \tilde{g} < \kappa_2^{-1} g \quad  \text{on $M$}$$
and that the sectional curvature of $\tilde{g}$ is bounded in absolute value by $\kappa_2^{-1}$, for a constant $\kappa_2$ only depending on $n$ and $\mathbf{A}$. Moreover, $\tilde{g}$ is $\mathbf{B}$-regular for $\mathbf{B}$ only depending on $n$ and $\mathbf{A}$ (to see this, one can consider covering spaces of bounded geometry for neighborhoods of points in $\partial U$). By taking $\kappa$ small enough compared to $\kappa_1,\kappa_2$, this metric $\tilde{g}$ satisfies the conclusion of the lemma.

\end{proof}

The next result extends the collapsing construction of Cheeger-Gromov \cite{CheegerGromov86}; it enables, in dimension 4, to further collapse already collapsed compact regions with a uniform curvature bound and without changing the metric near their boundaries. See  \cite[Theorem 4.1]{Rong93} for a related result. 
It is crucial for applications that the curvature bounds are ``uniform'', which means that they do not depend on the specific metric we start with, but only on its regularity and the dimension. 

\begin{prop} \label{keyprop}
Let $\epsilon_4>0$ be a small dimensional constant. Then for any $\mathbf{A}:=\{A_l\}_{l=0}^\infty$, there is $C=C(n,\mathbf{A})>0$, such that the following holds. 
Let $(M,g)$ be an $\mathbf{A}$-regular complete Riemannian $4$-manifold with  $|\sec_g|\leq 1$ and let $X\subset M$ be a compact $4$-dimensional domain with smooth boundary
such that  $\injrad_{g} \leq {\epsilon_4}$ on $X$. Then for any $\mu>0$ there is  a metric $\hat{g}$ on $M$ with
\begin{itemize}
\item $\hat{g}= g$ in a neighborhood of $M\setminus X$,
\item $|\sec_{\hat{g}}|\leq  C$,
\item $\Vol(\{x\in X;  \injrad_{\hat{g}}(x) \geq \mu\} ,\hat{g}) \leq C \Vol(T_{g,1}(\partial X), g) $.
\end{itemize}

\end{prop}

\begin{proof}

In this proof, we will use the words ``uniform'' or ``uniformly'' whenever a bound only depends on $n$ and $\mathbf{A}$. For instance a metric is uniformly regular if it is $\mathbf{B}$-regular for some $\mathbf{B}=\mathbf{B}(n,\mathbf{A})$. 

By a rescaling argument, and by applying Lemma \ref{chopping} to the $1$-neighborhood of $M\setminus X$,  it clearly suffices to prove the proposition in the special case where there is a $\kappa = \kappa (n,\mathbf{A})$ such that the $\kappa$-tubular neighborhood of $\partial X$ isometric to 
$$(\partial X \times [0, \kappa), g\big|_{\partial X} \oplus dt^2).$$
In that situation, $\partial X$ is a totally geodesic hypersurface, on which the induced metric has sectional curvature bounded between $-1$ and $1$. We set 
$$Y := X\setminus   T_{g,\kappa/3}(\partial X),\quad Z:= Y\setminus T_{g,\kappa/3}(\partial Y).$$ 
If $\epsilon_4$ is a small dimensional constant, $Y$ has an $F$-structure by \cite{CheegerGromov90} and our assumptions. 
Now, we want to deform the metric inside $X$ keeping the curvature \textit{uniformly} bounded and making the injectivity radius small far from $\partial X$. 
This does not follow from \cite[Sections 3 and 4]{CheegerGromov86} since there, the $F$-structure found is not uniformly controlled, and so the curvature bound depends on the starting Riemannian  manifold. For this reason, we need to make use of the more refined results of  \cite{CFG92}: to summarize the key point of the proof below, we apply the collapsing construction of \cite{CheegerGromov86} to the canonical $F$-structure induced by the $N$-structure found in \cite{CFG92}, and take advantage of the fact that this $N$-structure is uniformly controlled  \cite[Sections 6 and 7]{Fukaya88} \cite[Theorem 2.6, Section 8, Appendix 1]{CFG92}.


By \cite{CFG92} (see Subsection \ref{subsection:fn}) if $\epsilon_4$ is a small dimensional constant and if $\epsilon'$ another small dimensional constant  chosen in a moment, there is an $N$-structure $\mathfrak{n}$ on $Y$ which can be assumed to restrict to an $N$-structure on each slice $\partial X \times \{t\}$ in $T_{g,\kappa}(\partial X) \cap Y$ and a corresponding ``invariant'' $\mathbf{C}$-regular metric $g_{\epsilon'}$ close to $g$ and with bounded curvature on $X$, for some $\mathbf{C}=\mathbf{C}(n,\mathbf{A}, \epsilon')$. Because $g$ (resp.  $g_{\epsilon'}$) is $\mathbf{A}$-regular (resp. $\mathbf{C}$-regular) and $e^{-\epsilon'} g<g_{\epsilon'}<e^{\epsilon'} g$, if the dimensional constant $\epsilon'$ is small enough, we can construct a metric that interpolates between $g$ on $M\setminus X$ and $g_{\epsilon'}$ on $Y$ while keeping the curvature uniformly bounded, and which is a product metric on $Y\setminus Z$. This is done by generalizing \cite[Lemma 5.3]{CheegerGromov85}: using a convolution and center of mass argument, we first construct a map from $(\partial Y, g)$ to $(\partial Y,g_{\epsilon'})$ with uniformly bounded derivatives and we interpolate between these two metrics  using a cut-off function with uniformly bounded Hessian. (The smallness of $\epsilon'$ is to guarantee that the center of mass is well defined.) 
We will still call such a metric $g$, and we can assume that it is uniformly regular.

As explained in Subsection \ref{subsection:fn}, this $N$-structure $\mathfrak{n}$ gives rise to a canonical $F$-structure $\mathfrak{t}$ on $Y$. We can find a finite covering of $Y$ by open sets, called $\{V_\alpha\}$, where each open set  $V_\alpha$ is of the form $V(p_\alpha)$ for some point $p_\alpha\in Y$ as in the definition of $N$-structures. Moreover, we can assume that for points of $Y\setminus Z$, the open sets $V_\alpha$ are of the form $W_\alpha \times I$ where $W_\alpha$ are open subsets of $\partial X$ and $I$ are intervals of length $\kappa/3$.
Recall that orbits of $\mathfrak{n}$ contain torus quotients corresponding to the $F$-structure $\mathfrak{t}$.
 From the proofs in \cite{CFG92}, one checks that the $\{V_\alpha\}$ can be chosen so that each $V_\alpha$ is saturated by orbits of $\mathfrak{n}$ and so that $\{V_\alpha\}$ defines an atlas for the $F$-structure $\mathfrak{t}$ (the definition of atlas for an $F$-structure is recalled in Subsection \ref{subsection:fn}). 
Each $V_\alpha$ can be chosen to contain a $g$-geodesic ball of radius $\rho$ independent of $(M,g)$ (see Subsection \ref{subsection:fn}). Two important consequences are that $\mathfrak{t}$ has an atlas $\{V_\alpha\}$
with covering multiplicity bounded by a dimensional constant $c'_n$ (that is, any point is contained in at most $c'_n$ different $V_\alpha$),
and that each chart $V_\alpha$ has a uniformly large size depending only $n,\mathbf{A}$.
By the local description of $N$-structures in \cite[Appendix 1]{CFG92}, each $V_\alpha$ is covered by an open subset of the form 
$$T_R(\mathcal{O}^{\mathfrak{n}_\alpha}_{x_0}),$$
the tubular $R$-neighborhood of an orbit $\mathcal{O}^{\mathfrak{n}_\alpha}_{x_0}$ with normal injectivity radius at least $R$, where the radius $R\in (0,\kappa)$ depends only on $n,\mathbf{A}$. The second fundamental form of the orbit $\mathcal{O}^{\mathfrak{n}_\alpha}_{x_0}$ is uniformly bounded, see \cite[(8.7), Section 8]{CFG92}. The metric $g$ on that tubular neighborhood is uniformly close to the natural metric coming from the normal exponential map \cite[Section 8, Appendix 1]{CFG92}.

An $F$-structure $\mathfrak{t}$ on a $3$-manifold always has a polarized substructure $\mathfrak{t}'$, see \cite[Proposition 3.1]{Rong93b}. 
Moreover if $\{V_\alpha\}$ is an atlas of $\mathfrak{t}$, then $\{V_\alpha\}$ also naturally furnishes an atlas for $\mathfrak{t}'$. Hence, since $\dim(M)=4$, the $F$-structure on $Y$ considered above restricts to an $F$-structure on $\partial Y$ which has a polarized substructure $\mathfrak{t}'$.
By the uniform control described in the previous paragraph, by how $\mathfrak{t}$ is induced by $\mathfrak{n}$ (Subsection \ref{subsection:fn}), and by how $\mathfrak{t}'$ is constructed \cite[Proposition 3.1]{Rong93b},
the polarized $F$-structures $\mathfrak{t}'$ on the closed $3$-manifold $\partial Y$ can be chosen to be ``uniformly controlled'', so that the deformation of \cite[Sections 3]{CheegerGromov86} can be done with a uniform bound on the sectional curvature.


By the computations of \cite[Theorem 3.1]{CheegerGromov86} and the uniform control of $\mathfrak{t}'$, there is a family of metrics $h_{\delta'}$ on $\partial Y$ (with $\delta'\leq 1$) so that $h_1=g\big|_{\partial Y}$, the volume of $(\partial Y,h_{\delta'})$ is smaller than $C(\delta' )^{l_1} \log(\delta')^{l_2} \Vol(\partial Y, h_1)$ for some uniform constants $C,l_1,l_2>0$ and $|\sec_{h_{\delta'}}|$  is bounded by a constant $C=C(n,\mathbf{A})$ independent of $\delta'$. 
Let $\delta>0$ be a small constant to be fixed later.
Set $\bar{g}_{\delta}$ to be the metric on $Y\setminus Z$ defined by
$$\bar{g}_\delta := {h}_{e^{-\phi(r)}} \oplus dr^2 $$
on $\partial Y \times [0,-\log(\delta)]$, which is identified with $Y\setminus Z$, and where $\phi$ is a small perturbation of the identity which is constant on a neighborhood of $0$ and $-\log(\delta)$. 
Then the curvature of $\bar{g}_{\delta} $ is still upper bounded only in terms of $n$ and $\mathbf{A}$.

Let $0<\mu_1<\mu$ be two arbitrarily small constants. 
Using \cite[Section 4]{CheegerGromov86}, there are also metrics $g_{\delta'}$ on $Z$ (with $\delta'\leq 1$) with sectional curvature bounded by $C_1(n,\mathbf{A},g)$ perhaps depending on the metric $g$ but independent of $\delta'$, such that its injectivity radius is smaller than $\mu_1$ if $\delta'=\delta$ is small enough. 
Because of the metric product structure near  $\partial Z$ and the uniform control of $\mathfrak{t}'$ on $\partial Y \approx \partial Z$,
$g_\delta$ can be chosen to equal $h_\delta$ in a neighborhood of $\partial Z$.
Define $\tilde{g}_1$ on $X$ to be $g_\delta$ on $Z$, $\bar{g}_\delta $ on $Y\setminus Z$, and $g$ on $M\setminus Y$. 
Right now, the sectional curvature of $\tilde{g}_1$ is uniformly bounded on $M\setminus Z$,  but not necessarily on $Z$, so we need to make one last change of metric. For some large constant $A$, consider $Z\cup (\partial Z \times [0,A]) \cup (M\setminus Z)$, where we identify $\partial Z$ with $\partial Z\times \{0\}$, and $\partial Z\times \{A\}$ with $\partial (M\setminus Z)$ (this manifold is just $M$). Set $\tilde{g}_2$ to be the metric obtained by extending $\tilde{g}_1$ in the obvious way on $Z\cup (\partial Z \times [0,A]) \cup (M\setminus Z)$.
Since $\mu_1,\mu$ are independent, we can choose 
$$\mu_1 \ll C_1(n,\mathbf{A},g)^{-1} \mu.$$
Consider a function $F:M\to (0,\infty)$ with uniformly controlled Hessian and nonincreasing on $\partial Z \times [0,A] $, equal to $1$ on $M\setminus Z$, and equal to $C_1(n,\mathbf{A},g)$ on $Z$.
We set 
$$\hat{g}:=F\tilde{g}_2.$$
As $A$ can be taken arbitrarily large, we can choose $A$ and $F$ so that $|\sec_{\hat{g}}|\leq C(n,\mathbf{A})$, and $\hat{g}$ has  injectivity radius strictly less than  $\mu$ on $Z\cup (\partial Z \times [0,A])$. The volume $\Vol(\{x\in X;  \injrad_{\hat{g}}(x) \geq \mu\} ,\hat{g})$ is bounded as wanted by a simple computation.

\end{proof}


\subsection{A practical bound on $\essminvol$ in dimension 4}

As usual, if $(M,h)$ is a Riemannian $4$-manifold and $\hat{\epsilon}$ a dimensional constant, denote by $M^{(h)}_{> \hat{\epsilon}}$ the $\hat{\epsilon}$-thick part of $M$, that is:
$$M^{(h)}_{> \hat{\epsilon}} := \{x\in M; \quad \injrad_h(x) > \hat{\epsilon}\}.$$

\begin{prop} \label{quantify}
There are constants $\hat{\epsilon}>0$ and $C>0$ such that the following holds for any $T>0$ and any closed $4$-manifold $M$.
If there is a metric $g$ on $M$ with $|\sec_g|\leq 1$ such that
$$\Vol(M^{(g)}_{> \hat{\epsilon}},g) \leq T,$$
then 
$$\essminvol(M)\leq CT.$$

\end{prop}
\begin{proof}
This proposition is a corollary of Proposition \ref{keyprop} . Let $\hat{\epsilon} = \epsilon_4/2$, where $\epsilon_4$ is as in Proposition \ref{keyprop} . 
Let $(M,g)$ be as in the statement with 
\begin{equation} \label{tt}
\Vol(M^{(g)}_{> \hat{\epsilon}},g) \leq T.
\end{equation}
By \cite{Abresch88}, we can suppose that $g$ is $\mathbf{A}$-regular for some universal sequence $\mathbf{A}$ (we may need to replace $T$ by $2T$ in the inequality above).
In what follows, $C$ is a universal constant. Fix $\kappa\in (0,1)$. Let $\Omega\subset  M^{(g)}_{> \hat{\epsilon}}$ be an open subset with smooth boundary which is contained inside and closely approximates the region
$$ \{x\in M^{(g)}_{> \hat{\epsilon}}; \quad \dist_g(x, M^{(g)}_{\leq \hat{\epsilon}})>\kappa\}$$
and set 
$$X := M\setminus \Omega.$$
By \cite{CLY81,CGT82} we can and do fix $\kappa$ small enough (independently of $(M,g)$) so that 
$$X \subset M^{(g)}_{\leq \epsilon_4}.$$

Choose an arbitrarily small constant $\mu>0$.
We apply Proposition \ref{keyprop} to $X$ to obtain a metric $g'''_\mu$ on $M$ coinciding with $g$ on $\Omega=M\setminus X$, with $|\sec_{g'''_\mu}|\leq C$ and controlled volume in the sense that
$$\Vol(\{x\in X; \quad \injrad_{g'''_\mu}(x)> \mu\}, g'''_\mu) \leq C\Vol(T_{g,1}(\partial X),g).$$
But by Bishop-Gromov inequality, given a point $p\in \partial X=\partial \Omega$, the volume of $B_g(p,1)$ is controlled by the volume of $B_g(p,\kappa)$ which is contained inside $M^{(g)}_{> \hat{\epsilon}}$ by definition of $\Omega$. So (\ref{tt}) implies 
$$\Vol(T_{g,1}(\partial X),g) \leq C \Vol(M^{(g)}_{> \hat{\epsilon}},g)\leq CT,$$
so we get
$$\Vol(M^{(g'''_\mu)}_{>\mu}, g'''_\mu) \leq \Vol(\{x\in X; \quad \injrad_{g'''_\mu}(x)> \mu\}, g'''_\mu) + \Vol(\Omega,g) \leq CT.$$
We conclude by using the definition of $\essminvol(M)$.


\end{proof}

\vspace{3em}

\section{Essential minimal volume of Einstein $4$-manifolds} \label{section4}

\vspace{1em}

\subsection{Preliminaries} \label{marr}
We review some properties of Einstein $4$-manifolds $(M,g)$. 
In \cite{Anderson90,BKN89} it was proved that in any dimension $n$, a bound on the $L^{n/2}$-norm of the Riemann curvature tensor coupled with a non-collapsing condition imply that any limit space of a sequence of manifolds with bounded Ricci curvature is smooth away from finitely many points, at which the singularities are of orbifold type. Part of the proof hinges on an $\epsilon$-regularity theorem \cite{Nakajima88,Anderson89, Gao90}.
Given an Einstein $4$-manifolds $(M,g)$ ,
by the Chern-Gauss-Bonnet formula, the $L^2$-norm of the Riemann curvature tensor is given by the following formula:
\begin{equation} \label{cgb}
\int_M |\Rm_g|^2 dvol_g  = 8\pi^2 \e(M)
\end{equation}
where $ \e(M)$ is the Euler characteristic of $M$.
In \cite{CheegerNaber15}, Cheeger-Naber showed that a family of closed $4$-manifolds with bounded Ricci curvature, with a uniform bound on the diameter and volume, only admits finitely many diffeomorphism types. Though we will not need it, an $\epsilon$-regularity theorem for Einstein metrics holds even without non-collapsing assumptions by \cite{CheegerTian06}.  
\vspace{1em}

We will need a slight generalization of Cheeger-Naber's results for $4$-orbifolds with isolated singularities, the validity of which was explained to me by Aaron Naber. It states that a family of closed $4$-orbifolds with finitely many singularities, with $|\Ric|\leq 3$, uniformly non-collapsed and with a uniform bound on diameter, has only finitely many orbifold diffeomorphism types. In fact, Theorems 1.4 and 1.12 in \cite{CheegerNaber15} continue to hold for this class of orbifolds. One can see this thanks to a few elementary considerations. On a closed $4$-orbifold with finitely many singularities, most classical results of analysis are still true, for instance one can solve the Dirichlet problem on domains with the usual estimates on solutions, the heat kernel exists, Bishop-Gromov inequality holds, etc (see \cite{Farsi01} for instance). If we consider orbifolds with $|\Ric|\leq 3$ then one checks that properties like continuity of volume, metric cone theorems, stratification theorems for non-collapsed limit spaces are true by inspecting the proofs in \cite{Colding97,CheegerColding96,CheegerColdinga,CheegerColdingb,CheegerColdingc}.  (From a more general point of view, closed $n$-orbifolds with $\Ric\geq K$ as well as their limit spaces are  ``non-collapsed'' RCD$(K,n)$-spaces \cite{BKMR18}, so Cheeger-Colding's work generalizes to this setting by \cite{DePhilippisGigli18}.) 
In addition, by Bishop-Gromov's inequality and continuity of volume, the $\epsilon$-regularity theorem (2.8) in \cite{CheegerNaber15} still holds. These observations mean that the proof schemes of \cite{CheegerNaber15} carry over to our situation: for instance in order to show the analogue of Theorems 1.4 in \cite{CheegerNaber15}, it suffices to prove that there is no tangent space in limits spaces that splits off a factor $\mathbb{R}^{4-k}$ but not a factor $\mathbb{R}^{4-k+1}$, for $k=1,2,3$ . Since we assume the singular points to be isolated, for $k=1$ the proof of \cite[Theorem 6.1]{CheegerColdinga} works here, and similarly for $k=2,3$ the proofs of \cite{CheegerNaber15} are still valid.

\vspace{1em}

\subsection{The linear bound} \label{proofff}

In the main result of this section, we show that in dimension $4$, the essential minimal volume of Einstein manifolds can be estimated in terms of the topology. The bound we obtain does not need a non-collapsing condition and  leads to a conjectural topological obstruction for the existence of Einstein metrics on $4$-manifolds. The result can be seen as an extension of the diffeomorphism finiteness theorems \cite[Theorem 0.1]{AndersonCheeger91}, \cite[Theorem 1.12]{CheegerNaber15} to Einstein $4$-manifolds without diameter bound or non-collapsing assumptions.
\vspace{1em}

\begin{theo} \label{mainineq}
There is a constant $C>0$ so that for any smooth closed $4$-manifold $M$ admitting an Einstein metric, 
 \begin{align*}
 C^{-1} \e(M) \leq \essminvol(M) \leq C \e(M)
 \end{align*}
where $\e(M)$ is the Euler characteristic of $M$. 

\end{theo} 
\vspace{1em}

\subsection{Thick-thin decomposition for Einstein $4$-manifolds} \label{ttdecomposition}
In this subsection, we introduce in preparation for the proof of Theorem \ref{mainineq}, a thick-thin decomposition adapted to Einstein metrics. We will work with a small constant $v>0$ that will eventually be chosen independent of other parameters, and $C_v$ will denote a constant depending only on $v$, but which may change from line to line. Recall that for a metric, $|\sec|$ controls $|\Rm|$ and vice versa.

The lower bound in the statement of Theorem \ref{mainineq} is true generally, and follows from Lemma \ref{euler}. Thus, we only need to prove the upper bound.
Our proof will be based on a covering argument. Let $g$ be a metric on the closed smooth $4$-manifold $M$ with 
$$\Ric_g = \lambda g, \quad \lambda \in \{-3,0,3\} .$$  
Fix $v>0$ and $\epsilon>0$  small real numbers; $v$ will be chosen as a universal constant and $\epsilon$ depends only on $v$;  they will be fixed in Subsection \ref{end of proof}. 
It is useful to introduce a certain decomposition of $(M,g)$.
We define the collapsing scale $r_v$ as follows for any $p\in M$:
\begin{equation} \label{not1}
r_v(p):=\sup\{r\in(0,1]; \quad \Vol(B(p,r),g) \geq v r^4\}. 
\end{equation}
The function $r_v(.)$ is a positive continuous function.

\vspace{1em}

The thick-thin decomposition is defined as follows:
$$M= M^{[>\epsilon]}  \sqcup M^{[\leq \epsilon]} ,$$
where 
 \vspace{1em}
 \begin{equation} \label{not2}
 \begin{split}
M^{[>\epsilon]} & := \{x\in M; \quad \int_{B_g(x,\frac{r_v(x)}{24})} |\Rm_g|^2 dvol_g > \epsilon\},\\
 \vspace{1em}
M^{[\leq \epsilon]} & := \{x\in M; \quad \int_{B_g(x,\frac{r_v(x)}{24})} |\Rm_g|^2 dvol_g \leq \epsilon\}.
\end{split}
\end{equation}
\vspace{1em}


The choice of factor for the radius $\frac{r_v(x)}{24}$ is for technical convenience. If $g$ is Ricci-flat we \emph{always rescale it} so that its volume is less than $v$ and its diameter less than $1$. 

Then, regardless of the sign of the Ricci curvature, if we choose first $v$ small enough and $\epsilon$ small enough compared to $v$, one can check than at any point $x$ of the thin part, $M^{[\leq \epsilon]}$, we have $r_v(x)<1$, which in particular means by (\ref{not1}) that 
 \begin{equation} \label{not3}
 \Vol(B_g(x,r_v(x)),g) =v r_v(x)^n.
 \end{equation}
Moreover the injectivity radius on $M^{[\leq \epsilon]}$ is small compared to the curvature because by the $\epsilon$-regularity theorem \cite{Nakajima88,Anderson89,Gao90}: 
\begin{equation} \label{smol}
 \forall x\in M^{[\leq \epsilon]}, \quad \sup \{|\sec_x|^{1/2} ; \quad x\in B_g(x,\frac{r_v(x)}{48})\} \leq \frac{1}{r_v(x)} 
\end{equation}
and so near points of $M^{[\leq \epsilon]}$ the geometry is locally collapsed and carries an $N$-structure. This is how, given $v$, the constant $\epsilon$ is chosen with respect to $v$.

 
Let $\hat{\epsilon}>0$ be as in Proposition \ref{quantify}. 
By Chern-Gauss-Bonnet (\ref{cgb}) and Proposition \ref{quantify}, in order to prove the  upper bound in Theorem \ref{mainineq}, it suffices to show the following:

\begin{theo} \label{reduction}
The constants $v$, $\epsilon$ can be chosen small enough so that if $(M,g)$ is as above, there is a metric $\underline{\mathbf{g}}$ on $M$ with $|\sec_{\underline{\mathbf{g}}}|\leq 1$ and 
domains $\mathcal{W}\subset M$, $\mathcal{V}:=M\setminus \mathcal{W}$, such that
$$M^{[> \epsilon]} \subset \mathcal{W},$$
$$\mathcal{V} \subset M^{(\underline{\mathbf{g}})}_{ \leq \hat{\epsilon}},$$
$$\Vol(\mathcal{W}, \underline{\mathbf{g}}) \leq C\int_{M} |\Rm_g|^2 dvol_g$$
for some universal constant $C$.

\end{theo}

In the above statement, as usual, 
$$M^{(\underline{\mathbf{g}})}_{ \leq \hat{\epsilon}}:= \{x\in M;\quad \injrad_{\underline{\mathbf{g}}}(x)\leq \hat{\epsilon}\}.$$ 
To construct such a metric, we will roughly speaking first find a metric $\underline{\mathbf{g}}_1$ on a smooth approximation $\mathcal{W}$ of $M^{[> \epsilon]}$ with controlled volume, then a second metric $\underline{\mathbf{g}}_2$ on the complement $\mathcal{V}$ of $\mathcal{W}$ which is very collapsed, and we will glue these two metrics together.

\vspace{1em}

\subsection{Some covering lemmas}

If $B$ is a geodesic open ball and $\lambda$ a positive number, denote by $\lambda B$ the ball with same center and radius $\lambda$ times larger. Given a family of balls $\mathcal{F}$, we define its ``multiplicity'' to be the maximal number of different balls in $\mathcal{F}$ whose intersection is non-empty. For manifolds with bounded Ricci curvature, the usual Besicovitch covering lemma does not hold with a uniform constant, however the next lemma serves as a replacement in our case:

\begin{lemme} \label{recouvre}
There are finitely many points $x_1,...,x_{L'} \in M^{[>\epsilon]}$ such that 
$$M^{[> \epsilon]} \subset \bigcup_{i=1}^{L'}  B_g(x_i,\frac{r_v(x_i)}{3}) ,$$
and the multiplicity of the family $\{B_g(x_i,r_v(x_i))\}_{i=1}^{L'}$ is bounded from above by some $C_v$ depending only on $v$. 
\end{lemme}
\begin{proof}
Let $\tilde{\mathcal{F}}_k$ denote the set of  balls of the form $B_g(x,\frac{1}{6}r_v(x))$ with 
$$x\in M^{[> \epsilon]} $$ and radius $\frac{1}{6}r_v(x)$ in $(2^{-(k+1)},2^{-k}]$. For any $k\geq 0$, the set 
$$M^{[> \epsilon]} \cap \bigcup_{B\in \tilde{\mathcal{F}}_k} B  $$ 
can be covered by a family of balls $\tilde{G}_k \subset \tilde{\mathcal{F}}_k$ with uniformly bounded multiplicity by Bishop-Gromov (each point belongs to at most $c_{\text{univ}}$ balls of $\tilde{G}_k$ for a universal constant $c_{\text{univ}}$). Consider the family of balls $\bigcup_k \tilde{G}_k$, and remove from it any ball $B$ that is contained in $2B'$ for some ball $B'\neq B$ in $\bigcup_k \tilde{G}_k$. By doubling the size of each remaining ball, this yields a new family $\tilde{\mathcal{H}}$ of balls of the form $B_g(x,\frac{1}{3}r_v(x))$. Let $\tilde{\mathcal{H}}'$ be the family of balls $3B$, where $B\in \tilde{\mathcal{H}}$. Balls of $\tilde{\mathcal{H}}'$ are thus of the form $B_g(x,r_v(x))$.
\vspace{1em}

\textbf{Claim:} $\tilde{\mathcal{H}}$ is a covering of $M^{[> \epsilon]}$ and $\tilde{\mathcal{H}}'$ has bounded multiplicity.
\vspace{1em}

Let us check this claim, which would conclude the proof. The fact that $\tilde{\mathcal{H}}$ is a cover of $M^{[> \epsilon]}$ is clear from its construction and the identity $$M^{[> \epsilon]}  = \bigcup_k \Big( M^{[> \epsilon]} \cap \bigcup_{B\in \tilde{\mathcal{F}}_k} B \Big).$$ Suppose that the multiplicity of $\tilde{\mathcal{H}}'$ cannot be chosen uniformly bounded, i.e. it cannot be chosen bounded from above independently of $(M,g)$. Then for any integer $N>0$, there is such a $(M,g)$ such that the family $\tilde{\mathcal{H}}'$ constructed above has multiplicity at least $N$ at some point $x\in M^{[> \epsilon]} $,  in other words there are different balls $B_1,...,B_N\in \tilde{\mathcal{H}}'$ containing $x$. Let $y_1,...y_N$ be the respective centers of $B_1,...,B_N$. We assume without loss of generality that the radii are in non-decreasing order 
$$\rad B_1 \leq \rad B_2...\leq \rad B_N.$$ 
Since by construction, no $\frac{1}{6}B_i$ is contained in another $\frac{1}{3}B_j$ ($i\neq j$), by Bishop-Gromov inequality and by extracting a subsequence of $\{B_i\}$ if necessary, we can assume that 
$$  2^{-1}\dist_g(y_1,y_2)< 2^{-2} \dist_g(y_1,y_3)... < 2^{-(N-1)} \dist_g(y_1,y_N),$$
\begin{equation} \label{added3}
\rad B_1 <2^{-1} \rad B_2 ...< 2^{-(N-1)} \rad B_N,
\end{equation}
and moreover 
$$ C_v^{-1} \rad B_j \leq \dist_g(y_1,y_j) \leq C_v\rad B_j \quad \forall j\in \{1,..,N\}.$$
We can also make sure that $\rad B_{j_N} \to 0$ as $N\to \infty$.
But since 
$$r_v(y_j) =  \rad B_j,$$ $B_g(y_1,\dist_g(y_1,y_N))$ is not too collapsed:
$$\frac{\Vol(B_g(y_1,\dist_g(y_1,y_N)),g)}{\dist_g(y_1,y_{N})^4} \geq C_v^{-1}>0,$$
therefore by applying Bishop-Gromov monotonicity formula again, the difference of the volume ratios 
$$\frac{\Vol(B_g(y_1,\dist_g(y_1,y_{j_{N-1}})),g)}{\dist_g(y_1,y_{j_{N-1}})^4} - \frac{\Vol(B_g(y_1,\dist_g(y_1,y_{j_{N+1}})),g)}{\dist_g(y_1,y_{j_{N+1}})^4}$$
is arbitrarily small for some $j_N\in \{3,...,N-1\}$ if $N$ is chosen large enough. By a blow-up and compactness argument, by the ``volume cone implies metric cone" theorem \cite[Theorem 4.91]{CheegerColding96}, and since in dimension $4$ the singular set has codimension $4$ \cite[Theorem 1.4]{CheegerNaber15}, the ball 
$$\frac{1}{4} B_g(y_{j_N},\dist_g(y_1,y_{j_N}))$$ 
endowed with the metric $\dist_g(y_1,y_{j_N})^{-2} g$ would be arbitrarily close (in the $C^{\infty}$ topology by the Einstein condition) to a smooth open domain of a metric cone $\mathfrak{C}$ which is smooth outside its tip if $N$ is large enough. The smooth part of this $4$-dimensional cone $\mathfrak{C}$ is Ricci flat so actually it is flat. Since $y_{j_N}\in M^{[> \epsilon]} $, we have
$$\int_{B_g(y_{j_N}, \frac{1}{24} r_v(y_{j_N}))} |\Rm_g|^2 >\epsilon.$$
But for large $N$, since by construction $\frac{1}{6} B_g(y_1,r_v(y_1))$ is not included inside $\frac{1}{3}B_g(y_{j_N},r_v(y_{j_N}))$ and since  (\ref{added3}) holds, we have
$$  B_g(y_{j_N},\frac{1}{24}r_v(y_{j_N}))\subset   \frac{1}{4} B_g(y_{j_N},\dist_g(y_1,y_{j_N}))$$
so that would  eventually contradict the flatness of the cone $\mathfrak{C}$ as $N\to \infty$, by the $\epsilon$-regularity theorem.

\end{proof}

Let $\{x_i\}_{i=1}^{L'}$ be as in the previous lemma.
In following lemma we find annuli with bounded curvature inside $B_g(x_i,r_v(x_i))$.

\begin{lemme} \label{prep}
There are constants $\frac{1}{1000} >\delta_1>\delta_2>0$ and $C_v$ depending only on $v$ such that one can choose for each $i\in \{1,\dots,L'\}$ a radius 
$$r_i\in ((\frac{1}{2} - \delta_1)r_v(x_i),(\frac{1}{2}+\delta_1)r_v(x_i))$$
satisfying the following: at any point 
$$x\in B_{g}(x_i, r_i+\delta_2r_v(x_i)) \setminus B_{g}(x_i,r_i-\delta_2r_v(x_i)),$$
we have
\begin{equation} \label{curvaturebound}
r_v(x_i)^2 |\Rm_{g}(x)| \leq C_v,
\end{equation}
\begin{equation} \label{injradbound}
\injrad_g(x) \geq \delta_2r_v(x_i).
\end{equation}
\end{lemme}
\begin{proof}
We argue by contradiction: assume that there is a sequence of $4$-manifolds of the form $B_{g_k}(y_k,R_k)$ such that $g_k$ is a metric with $|\Ric_{g_k}|\leq n-1$, $R_k\leq 1$ and $\Vol_{g_k}(B_{g_k}(y_k,R_k)) \geq vR_k^4$, such that the conclusion of the lemma fails. By rescaling one can assume that $R_k=1$, and so we are supposing that one {cannot} find $\delta_1>\delta_2>0$ and $C_v$ such that for some radius
$$\bar{r} \in (\frac{1}{2}-\delta_1,\frac{1}{2}+\delta_1),$$
at any point 
$$x\in B_{g_k}(y_k,\bar{r} +\delta_2) \setminus   B_{g_k}(y_k,\bar{r} -\delta_2),$$
the bounds $|\Rm_{g_k}(x)|\leq C_v$ and $\injrad_{g_k}(x) \geq \delta_2$ hold.

By Cheeger-Naber \cite[Theorem 1.4, Theorem 8.6, Corollary 8.85]{CheegerNaber15} however, a subsequence of the pointed balls $(B_{g_k}(y_k,\frac{3}{4}), g_k,y_k)$ (which are all $v$-noncollapsed) converge in the Gromov-Hausdorff topology to a metric space $(X,g_\infty,y_\infty)$ smooth outside of finitely many points $u_1,...,u_K$. It is then possible to find $\delta_1>\delta_2>0$ such that  for some radius
$$\bar{r} \in (\frac{1}{2}-\delta_1,\frac{1}{2}+\delta_1),$$
the region
$$B_{g_k}(y_k,\bar{r} +\delta_2) \setminus   B_{g_k}(y_k,\bar{r} -\delta_2),$$
is $\delta_2$-far from $\{u_1,...,u_K\}$ as measured with the distance $g_\infty$. By the $\epsilon$-regularity theorem in \cite{Anderson90}\cite[(2.8)]{CheegerNaber15}, 
we get the desired bounds
$$|\Rm_{g_k}(x)|\leq C_v\text{  and  } \injrad_{g_k}(x) \geq \delta_2$$
on $B_{g_k}(y_k,\bar{r} +\delta_2) \setminus   B_{g_k}(y_k,\bar{r} -\delta_2)$
for $k$ large enough. This contradiction finishes the proof.

\end{proof}

Let $r_i$, $\delta_2$ be as in Lemma \ref{prep}.  By Lemma \ref{recouvre}, the balls $\{B_g(x_i,r_i - \frac{\delta_2}{2} r_v(x_i))\}_{i=1}^{L'}$ cover $M^{[>\epsilon]}$ and the covering multiplicity of $\{B_g(x_i,r_i )\}_{i=1}^{L'}$ is bounded by $C_v$. By removing some balls from this family, we can assume (after renumbering) that there is an integer $L\leq L'$ so that the points $x_1,...,x_L$ satisfy:
\begin{equation} \label{glissando}
\begin{split}
&  \text{$\{B_g(x_i,r_i - \frac{\delta_2}{2} r_v(x_i))\}_{i=1}^L$ cover $M^{[>\epsilon]}$ with covering multiplicity bounded by $C_v$,} \\
& \text{and for any $i\neq j\in \{1,...,L\}$, $B_g(x_i,r_i )$ is not included inside}\\
&  B_g(x_j,r_j - \frac{\delta_2}{2} r_v(x_j)).
\end{split}
\end{equation}  
Each $x_i$ ($i\in \{1,...,L\}$) is a point of $M^{[>\epsilon]}$, so $\int_{B_g(x_i,r_i -\frac{\delta_2}{2} r_v(x_i))} |\Rm_g|^2 >\epsilon$. Thus, as a consequence of the multiplicity bound on the covering $\{B_g(x_i,r_i -\frac{\delta_2}{2} r_v(x_i))\}_{i=1}^L$, there is $C_v>0$ so that
\begin{equation} \label{L}
L\leq C_v \int_{M} |\Rm_g|^2 dvol_g .
\end{equation}

The region $\bigcup_{i=1}^L B_{g}(x_i, r_i +  \frac{\delta_2}{2}r_v(x_i)) \setminus   \bigcup_{i=1}^L B_{g}(x_i, r_i -\frac{\delta_2}{2}r_v(x_i))$ is contained in $M^{[\le \epsilon]}$ and we want to cover it with nice balls capturing the collapsed geometry around that region:
\begin{lemme}\label{h-1}
There are finitely many points 
$$y_1,...,y_Q\in \bigcup_{i=1}^L B_{g}(x_i, r_i +  \frac{\delta_2}{2}r_v(x_i)) \setminus   \bigcup_{i=1}^L B_{g}(x_i, r_i - \frac{\delta_2}{2}r_v(x_i))$$ and corresponding radii $\bar{r}_1,...,\bar{r}_Q$ such that for some $C_v$ depending only on $v$:
\begin{enumerate}[label=(\roman*)]
\item $$\bigcup_{i=1}^L B_{g}(x_i, r_i +  \frac{\delta_2}{2}r_v(x_i)) \setminus   \bigcup_{i=1}^L B_{g}(x_i, r_i- \frac{\delta_2}{2}r_v(x_i)) \subset \bigcup_{q=1}^Q B_g(y_q,\frac{\bar{r}_q}{4}),$$
\item
$$  \frac{1}{48}r_v(y_q) \le \bar{r}_q \leq 1,$$
\item
$$\sup\{|\sec_x|^{1/2};\quad x\in B_g(y_q,\bar{r}_q)\} \le \frac{1}{\bar{r}_q} \text{ with equality if } \bar{r}_q<1,$$
\item
$$Q\leq C_vL.$$
\end{enumerate} 
Moreover for a universal constant $c_{\text{univ}}$, if $q,q'\in \{1,...,Q\}$ and
$$B_g(y_q,\frac{\bar{r}_q}{2}) \cap B_g(y_{q'},\frac{\bar{r}_{q'}}{2}) \neq \varnothing,$$ 
then a universal Harnack property holds: 
\begin{equation}\label{added2}
c_{\text{univ}}^{-1}\bar{r}_q \leq \bar{r}_{q'}\leq c_{\text{univ}}\bar{r}_q,
\end{equation}
and the multiplicity of the covering $\{B_g(y_q,\frac{\bar{r}_q}{2})\}_{q=1}^Q $ is bounded by $c_{\text{univ}}$.
\end{lemme}

\begin{proof}
This lemma follows from an argument appearing in \cite[Lemma 5.4, Theorem 5.5]{CheegerGromov85}. For 
$$x\in \bigcup_{i=1}^L B_{g}(x_i, r_i +  \frac{\delta_2}{2}r_v(x_i)) \setminus   \bigcup_{i=1}^L B_{g}(x_i, r_i -\frac{\delta_2}{2}r_v(x_i))$$
we set
$$\bar{k}(x) :=\inf \{\frac{1}{R} ; R\leq 1, \quad \sup_{x\in B_g(x,R)} |\sec_g(x)|^{1/2} \leq \frac{1}{R}\}.$$
In the statements and proofs of \cite[Lemma 5.4, Theorem 5.5]{CheegerGromov85} one should replace $\bar{k}$ by $\bar{k}^{-1}$ except in conformally changed metrics of the form ``$\bar{k}^2(.)g$''. Then their arguments imply our lemma, and give points $y_1,...,y_Q$ with radii
$$\bar{r}_q := \bar{k}(y_q)^{-1} \in(0,1]$$
so that the family of balls $\{B_g(y_q,\frac{\bar{r}_q}{4})\}_{q=1}^Q$ covers 
$$\bigcup_{i=1}^L B_{g}(x_i, r_i + \frac{\delta_2}{2}r_v(x_i)) \setminus   \bigcup_{i=1}^L B_{g}(x_i, r_i -\frac{\delta_2}{2}r_v(x_i)).$$
In particular, 
\begin{equation} \label{popo}
\sup\{|\sec_x|^{1/2};\quad x\in B_g(y_q,\bar{r}_q)\} \le \frac{1}{\bar{r}_q} \text{ with equality if } \bar{r}_q<1,
\end{equation}
and the multiplicity of $\{B_g(y_q,\frac{\bar{r}_q}{2})\}_{q=1}^Q$ is universally bounded (see proof of \cite[Theorem 5.5]{CheegerGromov85}). Moreover we have a Harnack property: the sizes of intersecting balls of the form $B_g(y_q,\frac{\bar{r}_q}{2})$ are universally comparable. 
The additional properties we need to check in order to finish the proof are: 
$$  \frac{1}{48}r_v(y_q) \le \bar{r}_q \leq 1$$
$$ \text{and }Q\leq C_vL.$$ 
The first line comes from (\ref{smol}).
By Lemma \ref{prep}, by (\ref{popo}) and the fact that $x_i\in M^{[>\epsilon]}$, if $B_g(y_q,\frac{\bar{r}_q}{4})$ intersects 
$$B_{g}(x_i, r_i +\frac{\delta_2}{2}r_v(x_i) \setminus   B_{g}(x_i, r_i - \frac{\delta_2}{2}r_v(x_i))$$ then $r_v(x_i)$ and $\bar{r}_q$ are comparable up to a factor $C_v$ and so Bishop-Gromov inequality yields the second line.

\end{proof}

\vspace{1em}

\subsection{Metric deformations}

Given $s>0$, given a  subset $S$ of a manifold $\mathcal{W}$ and a metric $h$ on $\mathcal{W}$, we recall the following notation for the $s$-neighborhood of $S$: 
$$T_{h,s}(S) := \{y\in \mathcal{W}; \dist_{h}(y,S) < s\}.$$
Given a compact Riemannian manifold with boundary $(\Omega,h)$, we say that $h$ is a product metric near $\partial \Omega\subset \Omega$ if there is a tubular neighborhood of $\partial \Omega$ isometric to 
$$\big(\partial \Omega \times [0,1), h|_{\Omega} \oplus dt^2 \big),$$
in particular $\partial \Omega$ is totally geodesic in that case.

Let $r_i$, $\delta_2$ be as in Lemma \ref{prep}. Let $\{x_i\}_{i=1}^L$ be as in (\ref{glissando}) and let $\{y_q\}_{q=1}^Q$, $\{\bar{r}_q\}_{q=1}^Q$ be as in Lemma \ref{h-1}. Recall that $\epsilon$ only depends on $v$ (see Subsection \ref{ttdecomposition}).

\begin{lemme} \label{h0}
There exists a universal constant $c_{univ}$ such that the following is true. 

For any $\eta>0$, if $v>0$ is taken small enough depending on $\eta$, then there is an open neighborhood $\mathcal{W}$ of $$\bigcup_{i=1}^L B_g(x_i,r_i)$$
and compact subdomains $\mathcal{R}_1,...,\mathcal{R}_{K_0} \subset \mathcal{W}$ with smooth boundaries
so that the following holds: there is a constant $C_v$ depending only on $v$ and a metric $h_0$ on $\mathcal{W}$ 
such that
\begin{enumerate}[label=(\roman*)]
\vspace{1em}

\item $\mathcal{W} \subset  \bigcup_{i=1}^LB_{g}(x_i,r_i) \cup  \bigcup_{q=1}^Q B_g(y_q,\frac{\bar{r}_q}{3})$,

 \vspace{1em}
\item for each $k\in \{1,...,K_0\}$, $h_0$ is a product metric near $\partial \mathcal{R}_k \subset \mathcal{R}_k$ and there is $i\in \{1,..., L\}$ with $\mathcal{R}_k \subset B_g(x_i,r_i)$,

\vspace{1em}
\item $\Vol(\mathcal{R}_k, h_0) \geq 1/C_v$, $\diam_{h_0} \mathcal{R}_k \leq C_v$ and $\injrad_{h_0} \geq 1/C_v$ on  $T_{h_0,1}(\partial \mathcal{R}_k)$ for all $k\in \{1,...,K_0\}$,

 \vspace{1em}
\item $|\Ric_{h_0}| \leq 3$ on $\mathcal{W}$, and $|\sec_{h_0}|\leq C_v$ on  $T_{h_0,1}\big(\mathcal{W} \setminus \bigcup_{k=1}^{K_0} \mathcal{R}_k\big)$,

 \vspace{1em}
\item
$ \Vol\big(T_{h_0,1}\big(\mathcal{W} \setminus \bigcup_{k=1}^{K_0} \mathcal{R}_k\big) , h_0 \big)\leq C_vL$,

\vspace{1em}
\item $h_0$ is a product metric near $\partial \mathcal{W} \subset \mathcal{W}$, $|\sec_{h_0}|\leq c_{univ}$ near $\partial \mathcal{W}$ and $\injrad_{h_0}(x) \leq \eta$ for $x\in \partial \mathcal{W}$. 

\end{enumerate}

\end{lemme}
\begin{proof}
Let $r_i$, $\delta_2$ be as in Lemma \ref{prep}. Let $\{x_i\}_{i=1}^L$ be as in (\ref{glissando}) and let $\{y_q\}_{q=1}^Q$, $\{\bar{r}_q\}_{q=1}^Q$ as in Lemma \ref{h-1}. As before, we denote by $C_v$ a constant depending only on $v$, and by $c_{\text{univ}}$ a universal constant. These constants can change from line to line.

Since $g$ is an Einstein metric and because of Lemma \ref{prep}, for all $i\in \{1,...,L\}$  and for all $x\in  B_{g}(x_i, r_i+\delta_2r_v(x_i)) \setminus B_{g}(x_i,r_i-\delta_2r_v(x_i))$:
\begin{equation} \label{pourg'}
\begin{split}
\injrad_{g}(x) &\geq \delta_2r_v(x_i) \geq \frac{r_i}{C_v}, \\
r_i^{k+1} |D^k_{g} \Rm_{g}| &  \leq C_{v}^{(k)} \text{ for all  }k\geq 0,
\end{split}
\end{equation}
where $D_{g}$ denotes the covariant derivative with respect to $g$. 
Similarly by Lemma \ref{h-1}, for all $q\in \{1,...,Q\}$ and all $y\in B_g(y_q,\bar{r}_q)$:
\begin{equation} \label{pourg''}
\bar{r}_q^{k+1} |D^k_{g} \Rm_{g}|  \leq c_{\text{univ}}^{(k)} \text{ for any  }k\geq 0.
\end{equation}


Note the following Harnack type property about the relative size of radii. Consider first $i\neq j \in \{1,...,L\}$ and suppose that 
$$B_{g}(x_i, r_i) \cap B_{g}(x_j, r_j) \neq \varnothing.$$
Then by (\ref{glissando}), 
the intersection of  
$$B_{g}(x_i, r_i) \setminus B_{g}(x_i,r_i - \frac{\delta_2}{2}r_v(x_i)) $$
and 
$$B_{g}(x_j, r_j) \setminus B_{g}(x_j,r_j - \frac{\delta_2}{2}r_v(x_j))$$
is non-empty. By the curvature bound of Lemma \ref{prep} and the definition of $ M^{[>\epsilon]}$, 
$$C_v^{-1} r_i \leq r_j \leq C_vr_i.$$

The analogous statement for the balls $\{B_g(y_q,\frac{\bar{r}_q}{2})\}_{q=1}^Q$ (with a universal constant) was shown in Lemma \ref{h-1} $(iv)$: if for $q,q' \in \{1,...,Q\}$
$$B_g(y_q, \frac{\bar{r}_q}{2}) \cap B_g(y_{q'}, \frac{\bar{r}_{q'}}{2}) \neq \varnothing,$$ 
then 
$$c_{\text{univ}}^{-1}\bar{r}_q \leq \bar{r}_{q'}\leq c_{\text{univ}}\bar{r}_q.$$

Finally, if $q\in \{1,...,Q\}$, $i\in \{1,.,L\}$ and if 
$$B_g(y_q, \frac{\bar{r}_q}{2}) \cap B_{g}(x_i, r_i ) \neq \varnothing,$$
then because $y_q$ is not in $\bigcup_{i=1}^L B_g(x_i,r_i - \frac{\delta_2}{2}r_v(x_i))$ by Lemma \ref{h-1}, we have 
$$B_g(y_q, \frac{\bar{r}_q}{2}) \cap \big(B_{g}(x_i, r_i) \setminus B_{g}(x_i,r_i - \frac{\delta_2}{2}r_v(x_i))\big) \neq \varnothing,$$
and so 
$$C_v^{-1} r_i \leq \bar{r}_q \leq C_vr_i.$$
Indeed the curvature bound of Lemma \ref{prep} and Lemma \ref{h-1} imply the first inequality. For the second inequality, recall that the curvature inside $B_g(y_q, {\bar{r}_q})$ is bounded by $\frac{1}{\bar{r}^2_q}$ (Lemma \ref{h-1}), and that $\int_{B_{g}(x_i, \frac{r_v(x_i)}{24})} |\Rm_g|^2 >\epsilon$ since $x_i \in M^{[>\epsilon]}$.

\vspace{1em}

We try to look for a conformal metric of the form ``$h_0= \rho^2 g$'' for some function $\rho$. Let us choose the conformal factor $\rho$ as follows. 
Let $\{\varphi_i\}_{i=1}^L$ and $\{\phi_q\}_{q=1}^Q$ be functions satisfying the following for some constant $C_v>0$ and universal constant $c_{\text{univ}}$:
$$\varphi_i,\phi_q:M \to [0,1],$$
$$\varphi_i=0 \text{ outside of $B_{g}(x_i, r_i ) $},  $$
$$\varphi_i=1 \text{ inside of $B_{g}(x_i,r_i  - \frac{\delta_2}{2}r_v(x_i))$},$$
$$\phi_q=0 \text{ outside of $B_{g}(y_q,\frac{\bar{r}_q}{2} ) $},  $$
$$\phi_q=1 \text{ inside of $B_{g}(y_q,\frac{\bar{r}_q}{3})$},$$

\begin{equation} \label{varphi}
\begin{split}
 r_i|D_{g}\varphi_i| \leq C_v, & \quad  r_i^{2}|\Hess_{g}\varphi_i| \leq C_v,\\
\bar{r}_q |D_{g}\phi_q| \leq c_{\text{univ}}, & \quad \bar{r}_q^{2}|\Hess_{g}\phi_q| \leq c_{\text{univ}}.
\end{split}
\end{equation}
where $\Hess_g$ denotes as usual the Hessian with respect to the metric $g$. The existence of such functions follows from the properties of $g$ listed at the beginning of the proof.

Set for all $i\in \{1,...,L\}$ or $q\in \{1,...Q\}$ and for any 
$$x\in  \bigcup_{i=1}^L B_g(x_i,r_i) \cup   \bigcup_{q=1}^Q B_g(y_q,\frac{\bar{r}_q}{3})$$
the following functions:
$$\psi_{i}(x):= \frac{\varphi_i(x)}{\sum_{i=1}^L \varphi_i(x) + \sum_{q=1}^Q \phi_q(x)} $$
$$\psi'_{q}(x):= \frac{\phi_q(x)}{\sum_{i=1}^L \varphi_i(x)+ \sum_{q=1}^Q \phi_q(x)} $$
$$\rho(x) := \sum_{i=1}^L \frac{\psi_i(x)}{r_i} + \sum_{q=1}^Q \frac{\psi'_q(x)}{\bar{r}_q}. $$
For any point 
$$x\in  \bigcup_{i=1}^L B_g(x_i,r_i) \cup   \bigcup_{q=1}^Q B_g(y_q,\frac{\bar{r}_q}{3}),$$ 
we have 
$$\sum_{i=1}^L \varphi_i(x) + \sum_{q=1}^Q \phi_q(x) \geq 1 .$$
Moreover, by the bound on the multiplicity of the families
$$\{B_g(x_i,{r_i})\}_{i=1}^L ,\quad  \{B_g(y_q,\frac{\bar{r}_q}{2})\}_{q=1}^Q,$$ 
(see Lemma \ref{recouvre}, Lemma \ref{h-1}), for any 
$$x\in  \bigcup_{i=1}^L B_g(x_i,r_i) \cup   \bigcup_{q=1}^Q B_g(y_q,\frac{\bar{r}_q}{3}),$$ 
 there is a small neighborhood $U_x$ of $x$ such that there are at most $C_v$ functions $\varphi_i$, and $c_{\text{univ}}$ functions $\phi_q$ which are non-vanishing in $U_x$. The Harnack property of the second paragraph ensures that the radii of nearby balls are comparable. Then observe that by (\ref{varphi}), the derivatives of  $\psi_i$ and $\psi'_q$ are locally bounded: at any point in $ B_g(x_i,r_i)$ or in $B_g(y_q,\frac{\bar{r}_q}{3})$, we have
$$|D_{g} \psi_i | \leq C_vr_i^{-1},\quad |D^2_{g} \psi_i | \leq C_vr_i^{-2},$$
$$\text{or } |D_g\psi'_q| \leq C_v \bar{r}_q^{-1},\quad |D^2_g\psi'_q |\leq C_v\bar{r}_q^{-2}$$
and in fact
$$|D_g\psi'_q| \leq c_{\text{univ}}\bar{r}_q^{-1},\quad |D^2_g\psi'_q |\leq c_{\text{univ}}\bar{r}_q^{-2} \text{ at points which are outside of $\bigcup_{i=1}^L B_g(x_i,r_i)$}.$$
Since $r_i\leq 1$, $\bar{r}_q\leq 1$,
at  any point in $ B_g(x_i,{r_i})$ or $B_g(y_q,\frac{\bar{r}_q}{3})$ we have again by the Harnack properties:
\begin{equation} \label{bahl}
\begin{split}
& C_v^{-1} r_i^{-1} \leq \rho \leq C_vr_i^{-1}, \\ 
& |D_{g} \rho| \leq C_vr_i^{-2}, \quad |D^2_{g} \rho| \leq C_vr_i^{-3}\\ 
&\text{or }\\
&  C_v^{-1} \bar{r}_q^{-1} \leq \rho \leq C_v\bar{r}_q^{-1}, \\
&  |D_{g} \rho| \leq C_v\bar{r}_q^{-2}, \quad |D^2_{g} \rho| \leq C_v\bar{r}_q^{-3}, \\
& \text{and in fact at points outside of $\bigcup_{i=1}^L B_g(x_i,r_i)$ we have:}\\
& c_{\text{univ}}^{-1} \bar{r}_q^{-1} \leq \rho \leq c_{\text{univ}}\bar{r}_q^{-1}, \\
&  |D_{g} \rho| \leq c_{\text{univ}}\bar{r}_q^{-2}, \quad |D^2_{g} \rho| \leq c_{\text{univ}} \bar{r}_q^{-3}.
\end{split}
\end{equation}

Consider the following metric: 
$$h_0 := \rho^2 g .$$
By the formula for the Riemann curvature tensor under a conformal change of metric, and by (\ref{bahl}), for any point in the union of 
$$\bigcup_{q=1}^Q B_g(y_q,\frac{\bar{r}_q}{3})$$
and 
$$\bigcup_{i=1}^L B_{g}(x_i, r_i) \setminus B_{g}(x_i, r_i -\frac{\delta_2}{2}r_v(x_i)),$$
the following curvature bound holds:
\begin{equation} \label{steuple}
|\sec_{h_0}(x)| \leq C_v.
\end{equation}
For points in 
$$  \bigcup_{q=1}^Q B_g(y_q,\frac{\bar{r}_q}{3}) \setminus \bigcup_{i=1}^L B_g(x_i,r_i) $$
the sectional curvature of $h_0$ is actually bounded by a universal constant $c_{\text{univ}}$ (in particular not depending on $v$), because the multiplicity bounds and Harnack property for balls in $\{B_g(y_q,\frac{\bar{r}_q}{2}) \}_{q=1}^Q$ are relying on a universal constant, see (\ref{bahl}).

Since by Lemma \ref{h-1} the balls in $\{B_g(y_q,\frac{\bar{r}_q}{4})\}_{q=1}^Q$ already cover
$$\bigcup_{i=1}^L B_g(x_i,r_i) \setminus B_{g}(x_i, r_i -\frac{\delta_2}{2}r_v(x_i)),$$ and because of (\ref{added2}), with respect to the metric $h_0$, the set
$$\bigcup_{i=1}^L B_g(x_i,r_i) \cup   \bigcup_{q=1}^Q B_g(y_q,\frac{\bar{r}_q}{3})$$
contains a complete $\frac{1}{c_{\text{univ}}}$-neighborhood of $\bigcup_{i=1}^L B_{g}(x_i, r_i) $.
Thus by the chopping theorem \cite{CheegerGromov91}, we can find a domain $\mathcal{W}$  such that
$$\bigcup_{i=1}^L B_g(x_i,r_i) \subset \mathcal{W} \subset \bigcup_{i=1}^L B_g(x_i,r_i) \cup \bigcup_{q=1}^Q B_g(y_q,\frac{\bar{r}_q}{3}),$$
and the boundary $\mathcal{W}$ has bounded second fundamental form. By Lemma \ref{chopping}, we modify the metric $h_0$ near $\partial \mathcal{W}$ so that the metric $h_0$ is now a product metric near $\partial \mathcal{W} \subset \mathcal{W}$ and the sectional curvature of $h_0$ is still universally bounded near $\partial \mathcal{W}$. 

By (\ref{smol}), the injectivity radius of $h_0$ near $\partial \mathcal{W}$ is less than the constant $\eta>0$ if $v$ is chosen small enough.

Similarly there is a neighborhood $\mathbf{U}$, with smooth boundary, of the set of points which are in 
$$\mathcal{W} \cap \bigcup_{i=1}^L B_g(x_i,r_i-\frac{\delta_2}{2}r_v(x_i))$$
 but not in
 $$ \bigcup_{i=1}^L B_g(x_i,r_i)\setminus B_g(x_i,r_i -\frac{\delta_2}{2}r_v(x_i)),$$
and whose boundary $\partial \mathbf{U}$ has bounded second fundamental form and the curvature is bounded by $C_v$ in a $\frac{1}{C_v}$-neighborhood neighborhood of it. We can again modify $h_0$ so that it is a product metric near $\partial \mathbf{U} \subset \mathbf{U}$. The bound (\ref{steuple}) means that 
$$\sup_{x\in \mathcal{W} \setminus \mathbf{U}} |\sec_{h_0}(x)|^{1/2} \leq C_v.$$

Recall that the curvature of $h_0$ near $\partial \mathcal{W}$ is bounded by a \textit{universal} constant, and we chose $v$ small enough so that the injectivity radius of $h_0$ near $\partial \mathcal{W}$ is less than $\eta$. (While the sectional curvature of $h_0$ near $\partial \mathcal{W}$ is bounded independently of the choice of $v$, the curvature bound elsewhere is potentially very large if $v$ is chosen small.)



Define $\mathcal{R}_1,...,\mathcal{R}_{K_0}$ to be the connected components of $ \mathbf{U}$. Note that by definition, each $\mathcal{R}_k$ is contained in a subset of the form $B_g(x_i,r_i)$. In our constructions, the constant $v$ is chosen according to $\eta$. Once $v$ is fixed, the constants $C_v$ is fixed. The properties in the statement of the lemma then follows from the constructions. For instance the volume bound of $T_{h_0,1}(\mathcal{W}\setminus \mathbf{U})$  comes from the covering multiplicity bounds, the Harnack property and the fact that $Q\leq C_vL$ (Lemma \ref{h-1}).

\end{proof}





Let $\mathcal{W}$, let $\mathcal{R}_1,...,\mathcal{R}_{K_0} \subset \mathcal{W}$, and let $h_0$ be given by Lemma \ref{h0}. Recall that by Lemma \ref{h0} $(iii)$ $(iv)$, each $(\mathcal{R}_k,h_0)$ has volume lower bound, diameter upper bound and $|\Ric_{h_0}|\leq 3$.
Note that by Lemma \ref{h0} $(ii)$ $(iii)$, there is a dimensional constant $C_v$ so that:
\begin{equation}\label{k0l}
K_0\leq C_vL
\end{equation}
where $K_0$ is the number of regions $\mathcal{R}_k$.

\begin{lemme} \label{ri}
There is a constant $C_v$ depending only on $v$, such that for each $\mathcal{R}_k$ ($k\in \{1,...,K_0\}$), one can find a metric $h_k$ which coincides with $h_0$ on a neighborhood of $\mathcal{R}_k$, such that
\begin{enumerate}[label=(\roman*)]
\item $|\sec_{h_k}|\leq C_v$ on $\mathcal{R}_k$,
\item $\injrad_{h_k}\geq C_v^{-1}$ on $\mathcal{R}_k$,
\item $\Vol(\mathcal{R}_k,h_k) \leq C_v$.
\end{enumerate}
\end{lemme}
\begin{proof}
We can argue by contradiction. Consider a sequence of Riemannian manifolds with boundary $\{(\hat{S}_j,g_j)\}_{j\geq 0}$ such that for some $c>0$,
\begin{itemize}
\item $T_{g_j,1} (\partial \hat{S}_j)$ has injectivity radius bounded below by $c^{-1}$ and sectional curvature bounded by $c$, 
\item $g_j$ is a metric product near $\partial \hat{S}_j \subset \hat{S}_j$,
\item $\diam_{g_j}(\hat{S}_j) \leq c$, $\Vol_{g_j}(\hat{S}_j)\geq c^{-1}$,
\item and $|\Ric_{g_j}|\leq 3$,
\end{itemize}
but such that the conclusion of the lemma fails. Then a subsequence converges in the Gromov-Hausdorff distance, and the proof of \cite[Theorem 1.12]{CheegerNaber15} shows that the possible diffeomorphism types of $\hat{S}_j$ is actually finite.
Hence, it is possible to put a metric on $\hat{S}_j$ with uniform control on curvature, injectivity radius and volume, which coincides with $g_j$ 
on $T_{g_j,{1}} (\partial \hat{S}_j)$.
 This is the desired contradiction.
\end{proof}

Let $\mathcal{V}$ be the complement of $\mathcal{W}$ in $M$. 
Given $\eta$, let $v$ and $\epsilon$ be fixed by Lemma \ref{h0}. We replace the metric $h_0$ of Lemma \ref{h0} by $h_i$ of Lemma \ref{ri} on each ${R}_i$ ($i\in \{1,...,{K_0}\}$), and call the new metric on $\mathcal{W}$ by $\underline{\mathbf{g}}_1$. Clearly, by Lemmas \ref{ri} and \ref{h0},  
$$|\sec_{{\underline{\mathbf{g}}_1}}|\leq C_v,$$
$$|\sec_{{\underline{\mathbf{g}}_1}}|\leq c_{univ} \text{ near $\partial \mathcal{W}$}$$ 
and the volume of $(\mathcal{W},\underline{\mathbf{g}}_1)$ is well controlled:
\begin{equation} \label{grooo}
\Vol(\mathcal{W},\underline{\mathbf{g}}_1) \leq C_vL.
\end{equation}
Moreover the injectivity radius of $\underline{\mathbf{g}}_1$ near $\partial \mathcal{W}$ is at most $\eta$

The next lemma produces a metric on $\mathcal{V}$, the complement of $\mathcal{W}$ in $M$.
\begin{lemme} \label{hoho}
There is a universal constant $c_{univ}$ such that the following is true.

Given $\eta>0$, there is a metric $\underline{\mathbf{g}}_2$ on a neighborhood of $\mathcal{V} =M\setminus \mathcal{W}$ such that 
\begin{enumerate}[label=(\roman*)]
\item $|\sec_{\underline{\mathbf{g}}_2}| \leq c_{univ}$,

\item $\injrad_{\underline{\mathbf{g}}_2} \leq \eta$ everywhere on $\mathcal{V}$,

\item $\underline{\mathbf{g}}_2 = \underline{\mathbf{g}}_1$ near $\partial \mathcal{V}$.

\end{enumerate}
\end{lemme}

\begin{proof}

Since $\mathcal{V}\subset M^{[\leq \epsilon]}$, similarly to the proof of Lemma \ref{h-1}, we get a cover of $\mathcal{V}$ by balls with controlled curvature relative to their size and satisfying a universal Harnack inequality. By a conformal modification of the metric as in the proof of Lemma \ref{h0}, if $v$ was chosen small enough, we get a metric $\underline{\mathbf{g}}_2$ on $\mathcal{V}$ of curvature bounded by a universal constant $c_{univ}$ and with injectivity radius at most $\eta$. In fact, we can make sure that $\underline{\mathbf{g}}_2 = \underline{\mathbf{g}}_1$ near $\partial \mathcal{V}$ by construction of $\underline{\mathbf{g}}_1$ near $\partial \mathcal{V}=\partial \mathcal{W}$.


\end{proof}

\vspace{1em}

\subsection{End of the proof} \label{end of proof}

\begin{proof}[Proof of Theorem \ref{reduction} and Theorem \ref{mainineq}]

Let $\eta>0$ and let $\underline{\mathbf{g}}_1$ and $\underline{\mathbf{g}}_2$ be the metrics obtained in the previous lemmas. For a large $A_v>0$, consider  $\mathcal{V} \cup (\partial \mathcal{V} \times [0,A_v])\cup \mathcal{W}$ where we naturally identify $\partial \mathcal{V}$ with $\partial \mathcal{V}\times \{0\}$  (resp. $\partial \mathcal{V}\times \{A_v\}$ with $\partial W$). This manifold is just $M$ itself.  We put the following metric $\underline{\mathbf{g}}_3$ on it:
$$\underline{\mathbf{g}}_3= \underline{\mathbf{g}}_1 \text{ on $\mathcal{W}$},$$
$$\underline{\mathbf{g}}_3= \underline{\mathbf{g}}_2 \text{ on $\mathcal{V}$},$$
$$\underline{\mathbf{g}}_3= \underline{\mathbf{g}}_1\big|_{\partial V} \oplus dt^2 \text{ on $[0,A_v]$}.$$
This smooth metric satisfies on $\mathcal{V} \cup (\partial \mathcal{V} \times [0,A_v]) $:
$$|\sec_{\underline{\mathbf{g}}_3}|\leq c_{univ},$$ 
$$\injrad_{\underline{\mathbf{g}}_3} \leq \eta,$$
and satisfies 
$$|\sec_{\underline{\mathbf{g}}_3}|\leq C_v$$ everywhere. 

Then, if $A_v$ was taken large enough, by conformally changing $\underline{\mathbf{g}}_3$ with a function with uniformly controlled Hessian, equal to $2c_{univ}$ on $\mathcal{V} $, and equal to a large constant on $\mathcal{W}$ depending only on $v$, we get a final metric $\underline{\mathbf{g}}$ with 
$$|\sec_{\underline{\mathbf{g}}}| \leq 1,$$
$$\injrad_{\underline{\mathbf{g}}} \leq \sqrt{2c_{univ}} \eta \text{ on $\mathcal{V}$}.$$
From (\ref{grooo}), we can also ensure that 
\begin{equation} \label{added5}
\Vol((\partial \mathcal{V} \times [0,A_v] ) \cup \mathcal{W},\underline{\mathbf{g}}) \leq C_v L
\end{equation}
for some $C_v$ depending only on $v$.

\vspace{1em}

Let $\hat{\epsilon}$ be as in Proposition \ref{quantify}. Set
\begin{equation} \label{added7}
\eta = \frac{\hat{\epsilon}}{2\sqrt{2c_{univ}}}
 \end{equation}
(now $v,\epsilon$ are finally fixed as in Lemma \ref{h0}).
The metric $\underline{\mathbf{g}}$ on $M$ constructed just above then satisfies $|\sec_{\underline{\mathbf{g}}}|\leq 1$ and 
\begin{equation} \label{added6}
 \injrad_{\underline{\mathbf{g}}} < \hat{\epsilon} \text{ on $\mathcal{V}$.} 
 \end{equation}
From our construction, and by respectively (\ref{added6}), (\ref{added5}) and (\ref{L}), we estimate
\begin{align*}
\Vol(M^{(\underline{\mathbf{g}})}_{> \hat{\epsilon}},\underline{\mathbf{g}}) 
& \leq C_v\Vol((\partial \mathcal{V} \times [0,A_v] ) \cup \mathcal{W},\underline{\mathbf{g}}) \\
& \leq C_vL \\
 & \leq C_v\int_{M} |\Rm_g|^2 dvol_g 
\end{align*}
where $C_v$ is a constant only depending on $v$, which was fixed by the choice of $\eta$ in (\ref{added7}),  independently of $(M,g)$.
Moreover (if we identify $(\partial \mathcal{V} \times [0,A_v])  \cup \mathcal{W}$ with $\mathcal{W}$) by Lemma \ref{recouvre}, Lemma \ref{h0} we have $M^{[>\epsilon]}\subset \mathcal{W}$ and by (\ref{added6}) $\mathcal{V} \subset M^{({\underline{\mathbf{g}}})}_{\leq \hat{\epsilon}}$. This proves Theorem \ref{reduction}. Combined with Lemma \ref{euler}, this  finishes the proof of Theorem \ref{mainineq}.

\end{proof}

\vspace{1em}

\subsection{An extension to $4$-orbifolds with isolated singularities} \label{extension}
In the next section, we will need the following generalization of Theorem \ref{mainineq} for closed $4$-orbifolds with isolated singular points. Let $M$ be a closed $4$-orbifold with finitely many singular points endowed with a smooth Einstein metric $g$. 
Let 
$\{p_1,...,p_K\}$
be the singular points of $M$, with corresponding groups $\Gamma_i \subset O(4)$, then for each $i=1,...,K$ there is a trivializing neighborhood $\hat{B}_i$ of $p_i$, that is $\hat{B}_i$ is a small geodesic ball of radius $\hat{r}_i$ centered at $p_i$ such that $(2\hat{B}_i, \frac{1}{\hat{r}_i^2}g)$ is, say, $C^2$-close to the quotient of the Euclidean ball ${B}_{\text{Eucl}}(0,2)$ by $\Gamma_i$.
By Chern-Gauss-Bonnet for compact $4$-manifolds with boundary, we have
$$\int_M |\Rm_g|^2  dvol_g \leq {c} \e(M)$$
$$\text{and } K\leq c \e(M)$$
for some universal constant $c$.
 
 Let $\hat{\epsilon}>0$ be as in Proposition \ref{quantify}.

\begin{theo} \label{gene}
Let $M$ be a closed $4$-orbifold with finitely many singular points $\{p_1,...,p_K\}$, endowed with a smooth Einstein metric $g$ as above. 
There is a metric $\underline{\mathbf{g}}$ on $$N := M\setminus  \{ p_1,...,p_K\}$$ such that 
\begin{enumerate}[label=(\roman*)]
\item $|\sec_{\underline{\mathbf{g}}}|\leq 1$,
\item near the end corresponding to $p_i$,  $\underline{\mathbf{g}}$ is a product metric of the form 
$$g_{\text{round}} \oplus dt^2$$
on $\partial \hat{B}_i \times [0,1)$ where $ g_{\text{round}}$ is the quotient of the round metric on $S^3$ with volume $8\pi$, and $dt^2$ is the standard metric on $[0,\infty)$,
\item the volume of the $\hat{\epsilon}$-thick part satisfies
$$\Vol(N^{(\underline{\mathbf{g}})}_{> \hat{\epsilon}}, \underline{\mathbf{g}}) \leq C\e(M)$$ where $C$ is a universal constant.
\end{enumerate}
\end{theo}

\begin{proof}
We only sketch the proof, since it is almost identical to that of Theorem \ref{mainineq}. Informally we treat the singular points as points with infinite curvature. Take $v,\epsilon>0$ small constants. 


As before, define for all $p\in M$:
$$r_v(p) := \sup\{r\in(0,1]; \Vol(B(p,r),g) \geq vr^4\}.$$ 
If for some $i_0\in \{1,...,K\}$, the limit volume ratio at $p_{i_0}$ satisfies
\begin{equation} \label{cuu}
\lim_{r\to \infty} \frac{\Vol(B(p_{i_0},r))}{r^4} \leq v,
\end{equation}
then there is a small region $U_{i_0}$ near $p_{i_0}$ which is a cone over a spherical quotient, and since $v$ is small, the geometry is locally collapsed and there is an $N$-structure on $U_{i_0}\setminus \{p_{i_0}\}$. We remove from $M$ small disjoint neighborhoods of all points $p_i$ such that (\ref{cuu}) holds at $p_i$. We obtain a compact orbifold with smooth boundary called $M'$, and we can make sure that a neighborhood of the boundary $\partial M'$ is saturated by an $N$-structure.
At any point $p\in M'$, $r_v(p)$ is strictly positive and we define 
\vspace{1em}

\begin{align*}
M'^{[> \epsilon]} := \{x\in M'; \quad & \int_{B_g(x,\frac{r_v(x)}{24})} |\Rm_g|^2 dvol_g >\epsilon \text{ or }  B_g(x,\frac{r_v(x)}{24}) \cap \{p_1,...,p_K\} \neq \varnothing \},
\end{align*}

$$ M'^{[\leq \epsilon]} := M'\setminus  M'^{[> \epsilon]}.$$
\vspace{1em}

By definition and by the $\epsilon$-regularity theorem, at any point $x$ of the thin part $M'^{[\leq \epsilon]}$, the ball $B_g(x,\frac{r_v(x)}{24})$ is smooth and locally collapsed, in particular it is covered by an $N$-structure. 
 
We can basically follow the proof of Theorem \ref{mainineq} step by step, applied to $M', M'^{[> \epsilon]}, M'^{[\leq \epsilon]}$. The only difference lies in the treatment of the singular points of $M'$. For this we can use the generalization of Cheeger-Naber \cite{CheegerNaber15} for orbifolds with isolated singularities described in Subsection \ref{marr}, which implies that the analogues of Lemmas \ref{recouvre}, \ref{prep} and \ref{ri} hold. In the proof of Lemma \ref{h-1}, we need to replace the definition of $\bar{k}$ by the following:
$$\bar{k}:=\inf\{\frac{1}{R}; R\leq 1, \quad \sup_{x\in B_g(x,R)} |\sec_g(x)|^{1/2} \leq \frac{1}{R} \text{ and } B_g(x,R)\cap \{p_1,...,p_K\}=\varnothing\}.$$

It follows from the arguments of the previous subsection that we get a metric on $M'$ with bounded curvature, such that the metric has bounded geometry near $\{p_1,...,p_K\} \cap M'$. Moreover the volume of the $\hat{\epsilon}$-thick part of $M'$ is bounded by  $C_v \e(M)$. Then after modifying once more that metric near $\{p_1,...,p_K\} \cap M'$ and $\partial M'$, we get the desired metric $\underline{\mathbf{g}}$ on $N=M\setminus \{ p_1,...,p_K\}$, which proves the theorem.
 \end{proof}

\vspace{1em}

\subsection{A conjectural obstruction and further remarks}
\label{crac}



Theorem \ref{mainineq} could be relevant in finding new topological obstructions for existence of Einstein manifolds (see \cite{Kotschick12} for a survey of the rare known obstructions). Indeed, we conjecture that for some constant $C>0$, for any surface $\Sigma_\gamma $ of genus $\gamma>1$ and any closed oriented $4$-manifold $N$, if $M$ is homeomorphic to the connected sum $(S^2\times \Sigma_\gamma) \sharp N$, one should have
$$\minvol (M) \geq \essminvol (M) \geq C \gamma.$$
If this is true and if $N$ is chosen so that $\e(M) < C'' \gamma$ for a small enough constant  $C''$, then $M$ would not carry any Einstein metric by Theorem \ref{mainineq}.

Theorem \ref{mainineq} is also related to a question of \cite{Kotschick98}, about the number of smooth structures carrying Einstein metrics, on a given topological $4$-manifold. 
What Theorem \ref{mainineq} implies is that, after removing regions saturated by $F$-structures,
there are only finitely many smooth structures carrying Einstein metrics on a given closed topological $4$-manifold.

We think that Theorem \ref{mainineq} should hold for $\minvol$ instead of $\essminvol$, but a proof would require a better understanding of the thin part of Einstein $4$-manifolds.


Theorem \ref{mainineq} was partly inspired by some analogous results which we proved in the context of minimal surfaces \cite[Corollary 3, Section 4]{moi19c}. There, the sheeted/non-sheeted decomposition, the Morse index and curvature bounds for stable surfaces replace respectively the thin-thick decomposition, the $L^2$-norm of the curvature and $\epsilon$-regularity theorems. The combinatorial argument of \cite{moi19c} could be adapted to give a proof of Theorem \ref{mainineq} not relying on the codimension 4 theorem of Cheeger-Naber \cite{CheegerNaber15}.

\vspace{3em}

\section{Essential minimal volume of complex surfaces} \label{section5}
This section is about the asymptotic behavior of the essential minimal volume on complex surfaces. The main theorem gives, for most complex surfaces, an asymptotic bound for $\essminvol$ in terms of topology. Before stating and proving this result, we review some elements pertaining to the classification of complex surfaces.

\vspace{1em}

\subsection{Preliminaries}
Closed complex surfaces are described via the Enriques-Kodaira classification \cite{BHPV04,FriedmanMorgan94}, which divides them according to their Kodaira dimension $\kappa\in \{-\infty,0,1,2\}$. For a complex surface $S$, let $K_S$ be its canonical line bundle. Let $P_m(S)$ be the dimension of the space of holomorphic sections of $K_S^m$, the $m$th tensor power of $K_S$. If $P_m(S)=0$ for all $m\geq1$, then set $\kappa:=-\infty$. Otherwise, set $\kappa:= \limsup_{m\to \infty} \frac{\log(P_m(S))}{\log(m)}$. A blow-up of $S$ corresponds in terms of diffeomorphism type to performing the connected sum $S\sharp \overline{\mathbb{C}P}^2$. The surface $S$ is called minimal if it is not the blow-up of another complex surface.

Surfaces of Kodaira dimension $2$ are said to be of general type. In some sense, most surfaces fall into that category (for instance smooth complete intersections of degree at least $5$ in $\mathbb{C}P^3$ are of general type). They have positive Euler characteristic.

Surfaces of Kodaira dimension $1$ or $0$ are all diffeomorphic to elliptic surfaces, and they have nonnegative Euler characteristic. Any minimal elliptic surface of Euler characteristic $0$ is diffeomorphic to a torus bundle over a surface on which one performs some logarithmic transforms. Any minimal elliptic surface is diffeomorphic to a fiber sum of an elliptic surface with Euler characteristic $0$ and copies of the rational elliptic surface (diffeomorphic to $\mathbb{C}P^2 \sharp 9\overline{\mathbb{C}P^2}$).

Sufaces of Kodaira dimension $-\infty$ are not completely classified. The K\"{a}hler case is understood: these are just blow-ups of $\mathbb{C}P^2$ or of $\mathbb{C}P^1$-bundles over a Riemann surface. Non-K\"{a}hler surfaces of Kodaira dimension $-\infty$ are called surfaces of class VII, and there are examples but no general classification so far.

\vspace{1em}

\subsection{The linear bound}
Quite similarly to real $2$-dimensional surfaces, the essential minimal volume of most complex surfaces is comparable to the Euler characteristic.

\begin{theo} \label{blip}
There is a universal constant $C>0$ such that for any closed complex surface $S$ of nonnegative Kodaira dimension,
$$C^{-1}\e(S)\leq \essminvol(S)\leq C \e(S).$$
\end{theo}

\begin{proof}
The lower bound is always true by Lemma \ref{euler}. 

Before continuing, note that since surfaces with nonnegative Kodaira dimension have nonnegative Euler characteristic, it is enough to prove the theorem for surfaces that are minimal by Lemma \ref{connected sum}. The upper bound is proved in two parts, depending on the Kodaira dimension of $S$. 

If $S$ is a minimal surface of Kodaira dimension $0$ or $1$, then it is diffeomorphic to an elliptic surface. If it has vanishing Euler characteristic, \cite[Theorem 5]{LeBrun99} implies that $\minvol(S)=\essminvol(S)=0$. Otherwise, $S$ is the fiber sum of a zero Euler characteristic elliptic surface with some rational elliptic surfaces and elementary considerations about collapsing theory suffice to show the desired linear bound, see the proof of \cite[Theorem 4]{LeBrun99}.
 
If $S$ is a minimal surface of general type, then its Euler characteristic is strictly positive. Moreover there is an orbifold K\"{a}hler-Einstein metric $g'$ on an orbifold $S'$ obtained from $S$ by contracting finitely many $(-2)$-curves  to points inside $S$ \cite{Aubin76,Yau78,Kobayashi85,Tsuji88}. 
One recovers $S$ from $S'$ by gluing some Ricci-flat ALE spaces \cite{Kronheimer86} $$(X_1,h_1),..,(X_K,h_K),$$ to the finitely many orbifold singularities $\{p_1,...,p_K\}$ of $S'$ (see also proof of Theorem 4.3 in \cite{LeBrun01}). These spaces are asymptotically conical, with limiting flat cone $(\mathcal{C}_i, g_{\text{Eucl}})$ and conical section a quotient of the round unit $3$-sphere. By using the Chern-Gauss-Bonnet formula and the fact that the $L^2$-norm of the Riemann curvature tensor on the non-flat Ricci-flat ALE spaces $X_i$ are lower bounded by a universal positive constant, we have
\begin{equation} \label{fff}
\begin{split}
& K  \leq \sum_{i=1}^K \e(X_i) \leq C \e(S) \\
& \text{and }\e(S')\leq \e(S) + K \leq C \e(S)
\end{split}
\end{equation}
for some universal constant $C>0$.

 Let $\hat{\epsilon}>0$ be as in Proposition \ref{quantify}.  By Theorem \ref{gene}, there is a metric $\underline{\mathbf{g}}$ on $N:=S'\setminus \{p_1,...,p_K\}$ with $|\sec_{\underline{\mathbf{g}}}|\leq1$, which is a nice product metric near each $p_i$, and whose $\hat{\epsilon}$-thick part $N^{(\underline{\mathbf{g}})}_{> \hat{\epsilon}}$ has controlled volume:
\begin{equation}  \label{all}
\begin{split}
\Vol(N^{(\underline{\mathbf{g}})}_{> \hat{\epsilon}} ,\underline{\mathbf{g}}) & \leq C \e(S') \\
  & \leq C\e(S)
 \end{split}
\end{equation}
where the universal constant $C$ changes from line to line.

Next, to fix ideas, we can imagine by rescaling that for each $i$ there is $x_i\in X_i$ such that $X_i\setminus B_{h_i}(x_i,1)$ is $C^2$-close to $(\mathcal{C}_i \setminus B_{g_{\text{Eucl}}}(O,1),g_{\text{Eucl}})$ where $O$ is the tip of $\mathcal{C}_i$. By the proof of Theorem \ref{mainineq}, for each $i=1,...,K$, there is  similarly a metric $\underline{\mathbf{g}}_i$ on $B_{h_i}(x_i,1)\subset X_i$ with $|\sec_{\underline{\mathbf{g}}_i}|\leq 1$, which is a nice product metric near the boundary $\partial B_{h_i}(x_i,1)$ and whose $\hat{\epsilon}$-thick part $\mathbf{B}_{i,\geq \hat{\epsilon}}$ satisfies:
\begin{equation}  \label{bll}
\Vol(\mathbf{B}_{i,\geq \hat{\epsilon}},{\underline{\mathbf{g}}_i}) \leq C \e(X_i).
\end{equation}
By gluing the new metrics $(X_i,\underline{\mathbf{g}}_i)$ back to $(S',\underline{\mathbf{g}})$ along their boundaries, we get the desired metric $\underline{\mathbf{h}}$ on $S$, since by (\ref{fff}), (\ref{all}), (\ref{bll}):
\begin{itemize}
\item $|\sec_{\underline{\mathbf{h}}}|\leq 1$, 
\item $\Vol(S^{(\underline{\mathbf{h}})}_{> \hat{\epsilon}}, \underline{\mathbf{h}})\leq C \e(S)$.
\end{itemize}
This concludes the proof thanks to Proposition \ref{quantify}.

\end{proof}




\bibliographystyle{alpha}
\bibliography{nouvelle_biblio_21_01_02}

\end{document}